\documentclass[12pt]{article}

\usepackage{amssymb,amsmath,amsthm,sectsty,url,
graphicx}
\usepackage[letterpaper,hmargin=1.0in,vmargin=1.0in]{geometry}
\usepackage[
colorlinks,linkcolor=black,citecolor=black,filecolor=black,urlcolor=black]{hyperref}
\usepackage{color}

\sectionfont{\large}
\subsectionfont{\normalsize}
\numberwithin{equation}{section}

\newtheorem{atheorem}{Theorem}

\newtheorem{maintheorem}{Theorem}
\newtheorem{theorem}{Theorem}[section]
\newtheorem{lemma}[theorem]{Lemma}
\newtheorem{proposition}[theorem]{Proposition}
\newtheorem{problem}[theorem]{Problem}
\newtheorem{corollary}[theorem]{Corollary}

\newtheorem{definition}[theorem]{Definition}
\newtheorem{remark}[theorem]{Remark}

\newtheorem{claim}[theorem]{Claim}

\renewcommand{\Pr}{ \mathrm P}

\newcommand{\eps}{\epsilon}

\newcommand{\vv}{\mathrm{v}}
\newcommand{\sss}{\mathrm{s}}
\newcommand{\con}{\mathrm{Cap}}
\newcommand{\X}{\mathbf{X}}

\newcommand{\Ind}[1]{\mathbf{1}\{#1\}}

\newcommand{\pd}{\partial}

\DeclareMathSymbol{\leqslant}{\mathalpha}{AMSa}{"36} 
\DeclareMathSymbol{\geqslant}{\mathalpha}{AMSa}{"3E} 
\DeclareMathSymbol{\eset}{\mathalpha}{AMSb}{"3F}     

\renewcommand{\epsilon}{\varepsilon}

\newcommand{\sfrac}[2]{\mbox{\small $\frac{#1}{#2}$}}
\newcommand{\ssfrac}[2]{\mbox{\footnotesize $\frac{#1}{#2}$}}
\newcommand{\half}{\ssfrac{1}{2}}

\newcommand{\N}{\mathbb N}
\newcommand{\R}{\mathbb R}
\newcommand{\Z}{\mathbb Z}

\usepackage[normalem]{ulem}

\begin{document}

\title{Recurrence of Markov chain traces}
\author{Itai Benjamini
\thanks{Department of Mathematics, Weizmann Institute
of Science, Rehovot 76100,
Israel.  e-mail: {\tt itai.benjamini@weizmann.ac.il}}
\and Jonathan Hermon
\thanks{
University of Cambridge, Cambridge, UK. E-mail: {\tt jonathan.hermon@statslab.cam.ac.uk}. Financial support by
the EPSRC grant EP/L018896/1.}
}

\date{}
\maketitle

\begin{abstract}
It is shown that transient graphs for the simple random walk do not admit a nearest neighbor transient Markov chain (not necessarily a reversible one), that crosses all edges with positive probability, 
while there is such chain for the square grid $\Z^2$. In particular,  the $d$-dimensional grid $\Z^d$ admits such a Markov chain  only when $d=2$. For $d=2$ we present a relevant example due to Gady Kozma, while the general statement for transient graphs is obtained by proving that for every transient irreducible Markov chain on a countable state space which admits a stationary measure, its trace is almost surely recurrent for simple random walk. The case that the Markov chain is reversible is due to Gurel-Gurevich, Lyons and the first named author (2007). We exploit recent results in potential theory of non-reversible Markov chains in order to extend their  result to the non-reversible setup. 
\end{abstract}



\section{Introduction}
In this paper we prove that for every transient nearest neighbourhood Markov chain (not necessarily a reversible one)    on a graph, admitting a stationary measure, the (random)  graph  formed by the vertices it visited and edges it crossed is a.s.\ recurrent for simple random walk. We use this to give partial answers to the following question: Which graphs admit  a  nearest neighbourhood transient  Markov chain that crosses all
edges with positive probability? 

\medskip

We start with some notation and definitions. 
Let $G=(V,E)$ be a  connected graph\footnote{We allow it to contain loops, i.e.~$E \subseteq \{\{ u,v\}:u,v \in V \} $.}. We say that $G$ is \emph{locally finite} if each vertex has a finite degree.
A \emph{simple random walk} (\textbf{SRW}) on $G$ is a (reversible) Markov chain  with state space $V$ and transition probabilities: $P(x,y)=\Ind{x \sim y}/\deg(x)$, where $x \sim y$ indicates that $\{x,y\} \in E$. Some background on Markov chains will be given in \S\ref{s:Markov} (for further background on Markov chains see, e.g.~\cite{aldous,levin,lyons}).

We call  $(V, E,c)$ (or $(G,c)$) a \emph{network}  if $G=(V,E)$ is a graph
and $c=(c_{e})_{e
\in E}$ are symmetric edge weights (i.e.~$c_{u, v}=c_{v,u} \ge 0 $ for every $\{u,v \} \in E$ and $c_{u,v}=0$ if $\{u,v\} \notin E$). We say that the network is \emph{locally finite} if $c_v:=\sum_{w} c_{v, w}< \infty $ for all $v \in V $. We emphasize that the network $(G,c)$ can be locally finite even when the graph $G$ is not locally finite. We shall only consider locally finite networks (even when this is not written explicitly). The  (weighted) nearest neighbor   random walk corresponding to    $(V, E,c)$ repeatedly does the following: when the current state is $v\in V$, the walk will move to vertex $u$  with probability $c_{u, v}/c_{v}$.
The  choice  $c_{e}=\Ind{e \in E }$ corresponds to SRW  on $(V,E)$. We denote this network by $(V,E, \mathbf{1} )$.

Let $\mathbf{X}:= (X_n)_{n \in \N }$ be an irreducible Markov chain on a countable state space $V$. Denote its transition kernel by $P$. Let $\pi$ be a stationary measure.   Consider the (undirected) graph supporting its transitions \begin{equation}
\label{e:G(X)}
G(\X):=(V,E) \quad \text{where} \quad E:=\{ \{u,v\}:P(u,v)+P(v,u)>0 \}.\end{equation} Note that the graph $G(\X)$ need not be locally finite and that it is possible that $P(x,y)=0$ for some $\{x,y\} \in E$. Harris \cite{Harris} proved that under a mild condition  $\X$ admits a stationary measure. It follows from his general condition that if each state $x $ can be accessed in one step from only finitely many states (i.e.\ $|\{y:P(y,x)>0 \}|<\infty$ for all $x$; note that this is a weaker condition than $G(\X)$ being locally finite) then $\X$ admits a stationary measure. Fix an initial state $o$. We say that $\mathbf{X} $ is \textbf{recurrent} if $o$ is visited infinitely often a.s.. Otherwise,  $\X$ is said to be \textbf{transient}.
 Let \[N(x,y)=N(y,x):=| \{n:  \{X_n,X_{n+1}\}= \{x,y  \} \} |  \] be the number of (undirected) crossings of the edge $\{x,y \}$. Let $\mathrm{PATH} $ be the graph formed by the collection of states visited by the chain  $V(\mathrm{PATH}):=\{X_n : n \in \Z_+  \} $ and the collection of edges crossed by the chain (in at least one direction) \[E(\mathrm{PATH}):=\{e \in E:N(e)>0 \}=\{ \{X_n,X_{n+1}\} : n \in \Z_+  \},\] where $\Z_+$ is the set of non-negative integers.  We show that if $\mathbf{X}$ is transient (and admits a stationary measure) then a.s.\ $\mathrm{PATH} $ forms a (random) recurrent subgraph of $G$ for simple random walk. In fact, the same is true even when the edges are
weighted by the number of crossings or by their expectations. This is completely analogous to the main result from \cite{traces} apart from the fact that we do not assume $\X$ to be reversible, nor $G(\X)$ to be locally finite. Paraphrasing \cite{traces}, the way to interpret Theorem \ref{thm:1} is that the trace of transient Markov chains (admitting a stationary measure) is in some sense ``small".
\begin{maintheorem}
\label{thm:1}
Let $\mathbf{X}$ be a transient irreducible Markov chain, admitting a stationary measure, on a countable state space $V$. Let $\mathrm{PATH},N $ and $G(\X)$ be as above. Then the following hold:
\begin{itemize}
\item[(i)]
The (random) graph $\mathrm{PATH}$ is a.s.\ recurrent for SRW.
 \item[(ii)]
 The random walk on the (random) network $(\mathrm{PATH},N)$ (in which the edge weight of an edge $e $ is $N(e)$) is a.s.\ recurrent.
\item[(iii)] For every $o \in V$ the random walk on the network $(G(\X),\mathbb{E}_o[N]) $ (in which the edge weight of an edge $e $ is $\mathbb{E}_o[N(e)]$) is recurrent.
\end{itemize}
\end{maintheorem}

\begin{remark}
As noted above, if $|\{y:P(y,x)>0 \}|<\infty$ for all $x$ then the chain admits a stationary measure. In \S\ref{s:relax} we relax the condition that $\X$ admits a stationary measure. 
\end{remark}

Theorem \ref{c:main} is obtained as an immediate corollary of Theorem \ref{thm:1}.
\begin{maintheorem}
\label{c:main}
Let $G=(V,E)$ be some connected locally finite graph which is transient for SRW. Then there does not exist an irreducible transient Markov chain $\X$ with $G(\X)=G $ which with positive probability crosses each edge (in at least one direction) at least once. In fact, by Rayleigh's monotonicity principle, there does not exist an irreducible transient Markov chain $\X$ admitting a stationary measure such that $G$ is a subgraph of $G(\X) $, which with positive probability crosses every edge of $E$ (in at least one direction) at least once.
\end{maintheorem}
One may wonder for which graphs $G$ there exists an irreducible transient Markov chain $\X$ with $G(\X)=G $ which with positive probability crosses each edge at least once (see \S\ref{s:open} for several related open questions). The following theorem gives a complete answer in the case that $G$ is the $d$-dimensional lattice $\Z^d$.

\begin{maintheorem}
\label{thm:3}
There exists an irreducible transient Markov chain $\X$ with $G(\X)=\Z^d$ which with positive probability crosses all edges (in at least one direction) iff $d=2$.    
\end{maintheorem}
The case $d \ge 3$ follows from Theorem \ref{c:main}. The case $d=1$ follows from the simple observation that a transient Markov chain on $\Z$ will eventually remain from some moment on either non-negative, or non-positive. A similar argument works for any graph $G$ with more than one end. The example for the case $d=2$ is due to Gady Kozma. We are tremendously grateful to him for allowing us to present it.  At the heart of the example there is a family of auxiliary transient birth and death chains (also due to Kozma) which is in some asymptotic sense ``as recurrent as possible"  for a transient Markov chain.

Recall that a birth and death chain (on $\Z_+$) is a chain with state space $\Z_+$ such that $P(x,y)=0 $ if $|x-y|>0$ and $P(x,y)>0$ whenever $|x-y|=1$. We recall that by Kolmogorov's cycle condition such a chain is always reversible and thus can be represented as a network.  
\begin{maintheorem}
\label{thm:4}
For every transient birth and death chain (on $\Z_+ $)  a.s.\ $ \sum_{i \in \Z_+} \frac{1}{N(i,i+1)}= \infty$. Conversely, for every $s$ there exists a birth and death chain such that for some constant $c(s)>0$, with positive probability we have that $N(i,i+1) \ge c(s) g_s(i) $ for all $i \in \Z_+$, where
\begin{equation}
\label{e:logloglog}
\begin{split}
& g_s(i):=i \prod_{j=1}^{s}\log_*^{(j)}i,  \\ &  \log_*^{(j)}i:=\begin{cases}\log^{(j)}i, & \text{if }\log^{(j-1)}i \ge e \\
1, & \text{otherwise} \\
\end{cases}
\end{split}
\end{equation}
and $\log^{(j)}:= \log \circ \log^{(j-1)} $ is the $j$th iterated logarithm.
\end{maintheorem}
The claim that for every transient birth and death chain \ $ \sum_{i \in \Z_+} \frac{1}{N(i,i+1)}= \infty$ a.s.\ follows from Theorem \ref{thm:1} (part (ii)) (or rather, from \cite{traces} as birth and death chains are always reversible). Indeed, for a birth and death chain with edge weights $(c(i,i+1):i \in \Z_{+}) $, the effective-resistance from $0$ to $\infty$ is $\sum_{i \ge 0}\sfrac{1}{c(i,i+1)} $. The sum diverges iff the chain is recurrent. 
\subsection{Comments}
\label{s:comments}
\begin{enumerate}
\item
As noted in \cite{traces} (see \eqref{e:concave} below) part (iii) of Theorem \ref{thm:1} implies part (ii). In turn, part (ii) of Theorem \ref{thm:1} implies part (i). To see this, first note that in terms of recurrence/transience of the random walk started at $o$, the network $(\mathrm{PATH},N)$ is equivalent to $(G(\X),N)$.  As SRW on  $\mathrm{PATH}$ is the random walk on  $(\mathrm{PATH}, \mathbf{1} )$,  Rayleigh's monotonicity principle (e.g.~\cite[Ch.\ 2]{lyons} or \cite{doyle}), together with the fact that $N \ge 1$ on the edges of $\mathrm{PATH} $, implies that if  $(\mathrm{PATH},N)$ is recurrent then so is SRW on  $\mathrm{PATH}$.
\item
Even for $G=\Z^3$ it is not hard to construct an irreducible Markov chain $\X$ with $G(\X)=\Z^3 $ which visits every vertex with positive probability. This is true for any graph which has a spanning tree which is isomorphic to the infinite path $(\N,\{\{z,z+1 \}:z \in \N  \})$.
\item
If we only consider Markov chains $\X$ such that $\sup_{y \in V}\pi(y)<\infty $, where $\pi$ is the stationary measure of $\X$, then one can prove the assertion of Theorem \ref{c:main} directly, without relying on Theorem \ref{thm:1}. Namely, by \eqref{e:N} in conjunction with part (i) of Lemma \ref{lem:v} we have that $\inf_{y \in V }\sum_{i=0}^{\infty}P^{i}(x,y)=0 $ for all fixed $x \in V $ (in the setup of Theorem \ref{c:main} under the additional assumption that  $\sup_{y \in V}\pi(y)<\infty $). Picking a sequence $\mathbf{y}:= (y_j)_{j=1}^{\infty}$ such that $\sum_{i=0}^{\infty}P^{i}(x,y_{j})<2^{-j} $ and applying the Borel-Cantelli Lemma we get that a.s.\ only finitely many vertices in the sequence $\mathbf{y}$ are visited. We believe that for the general case (when one allows $\sup_{y \in V}\pi(y)=\infty $) there does not exist an alternative simple proof to Theorem \ref{c:main}. Indeed, by Theorem \ref{thm:3} such an argument must fail for $\Z^2$, which suggests it has to be more involved. In fact, we do not have a direct (simpler) proof of Theorem \ref{c:main} (which does not rely on Theorem \ref{thm:1}),
even when further assuming $G$ is non-amenable with one end (e.g.\ a co-compact lattice in real hyperbolic spaces $\mathbb{H}^d$ or the product of a $d$-regular tree with $\Z$).
\item
For every graph $G$ the Doob's transform of the SRW corresponding to conditioning on returning to some fixed vertex is always recurrent and hence a.s.\ crosses all edges.
\item
Clearly every transient birth and death chain  with state space $\Z_+$ started from 0 crosses every edge. 
\item
By taking $\X$ to be the product of the birth and death chain from Theorem \ref{thm:4} (with any value of $s$)  with a complete graph of size $k$ we get a transient instant
with $G(\X)$ having minimal degree $k$ such that with positive probability $\X$ crosses every edge of $G(\X)$. In fact, one can consider a variant of the construction in which the degrees are not bounded. For $s=2$ from Theorem \ref{thm:4} it is not hard to verify that one can replace the vertex $k$ of the birth and death chain by a clique of size $\lceil \sqrt{k+1} \rceil $ for all $k$ (as described below) while keeping the property that with positive probability each edge is crossed at least once. More precisely,   for all $k \ge 1$ in the $k$th clique exactly one vertex is connected to a single vertex of the $(k+1)$th and  $(k-1)$th cliques with edge weights $w(k,k+1) $   and $w(k,k-1)$, respectively, where $w(\bullet,\bullet)$ are the edge weights from the case $s=2$). The  edge weights of edges between two vertices of the $k$-th clique can be taken to be $w(k,k+1)$.  
\end{enumerate}
\subsection*{Open problems}
\label{s:open}
\begin{problem}
Find a natural condition on a connected, locally finite, recurrent (for SRW) graph $G$  which ensures the existence of a transient Markov chain $\X$ with $G(\X)=G $ which with positive probability crosses each edge of $G$ at least once?
\end{problem}
One possible natural condition is that the graph is one ended. One can also consider concrete examples like the Uniform Infinite Planar Triangulation. 

Recall that transience is invariant under rough isometries (e.g.~\cite{lyons} Theorem 2.17).
\begin{problem}
Is the property of admitting a nearest neighbor Markov chain which with positive probability crosses all edges at least once  invariant under quasi-isometries?
\end{problem}
\begin{problem}
Show that there does not exist a transient nearest neighbor \textbf{reversible} Markov chain on $\Z^2$ whose transition probabilities are uniformly bounded from below, which with positive probability crosses all edges. See comments (2), (3) and (6) in \S\ref{s:conc} for details concerning some difficulties in constructing such example.
\end{problem}
Similarly, one can impose reversibility also to Problems 1.2-1.3. Moreover, one can impose also the condition that if $G(\X)=G=(V,E) $ then there exists some $c>0$ such that for every $x,y \in V $ such that $\{x,y\} \in E$ we have that $P(x,y) \ge c $. For more on this point see  comment (4) in \S\ref{s:conc}. 
\begin{problem}
Can one prove a quantitative version of Theorem \ref{thm:1} in the spirit of \cite[Theorem 10]{cut3}? 
\end{problem}
\begin{problem}
Is it the case that the trace of a transient branching Markov chain\footnote{A branching Markov chain evolves as follows: at each time unit each particle splits into a random number of offspring particles (according to some prescribed distribution) which all make one step independently of each other, chosen according to the transition kernel of the underlying Markov chain.} (not necessarily a reversible one) is recurrent w.r.t.\ branching simple random walk with an arbitrary offspring distribution whose mean is larger than 1.
\end{problem}
\begin{problem}
Consider the simple random walk snake. That is, SRW on a graph which is indexed
by a family tree of critical Galton-Watson tree conditioned to survive.
See e.g.\ \cite{BC}.  
Is the snake range  on any graph a.s.\ recurrent for the snake?
\end{problem}
\subsection{Related work}
\label{s:related}
The review of related works below is for the most part borrowed from \cite{traces}.
The assertion of Theorem \ref{thm:1} was first proved for SRW on the Euclidean lattice $\Z^d$, for which it  is known that for $d \ge 3$ the path performed by the walk
has infinitely many cut-times, a time when the past of the path is
disjoint from its future; see \cite{NP,cut,cut3}.
From this, recurrence
follows by the criterion of Nash-Williams \cite{NW} (see e.g.\ \cite[(2.14)]{lyons}). By contrast, James, Lyons and Peres
\cite{JLP}
constructed a transient birth and death chain which
a.s.\ has only finitely many cut-times. In fact, their example is the same as the one from Proposition \ref{p::bd1}, which is part of the proof of Theorem \ref{thm:4}.  Benjamini, Gurel-Gurevich and Schramm  \cite{cut3} constructed   a bounded degree graph satisfying that the trace of SRW on it a.s.\ has a finite number of cut-points. They also showed that  every transient Markov chain satisfies that the expected number of cut-times is infinite.   $ $

The case of Theorem \ref{thm:1} when $\mathbf{X}$ is reversible was proved by Benjamini, Gurel-Gurevich and  Lyons \cite{traces}.\footnote{Another proof for the case that $\X$ is a SRW on a locally finite graph can be found in \cite{BGG}.}\footnote{The proof in \cite{traces} contains a (minor) gap which was filled in \cite[Ch.\ 9]{lyons}. In fact, Lemma 9.23 in \cite{lyons}, used to fill the aforementioned gap will be crucial in our analysis.}  We follow their argument closely, adapting certain parts of it to the non-reversible setup.
We note that a priori it is not clear that their argument can be extended to the non-reversible setup because it relies on the connections between reversible Markov chains and electrical networks\footnote{While this connection was recently extended to the non-reversible setup by Bal\' azs and Folly \cite{net}, we did not manage to use their framework for the purpose of extending the argument from \cite{traces} to the non-reversible setup.} and also on the Dirichlet principle for reversible Markov chains (see \eqref{e:DP} below).

The connection between reversible Markov chains,  electrical networks and potential theory is classical (see e.g.\ \cite{doyle} or \cite[Ch.~2]{lyons} for a self-contained introduction to the topic). It was only in recent years that this connection was extended to the non-reversible setup in several extremely elegant works. The first progress on the non-reversible front was made by Doyle and Steiner \cite{doyle2} who derived an extremal characterization of the commute-time between two states, which shows that commute-times in the additive symmetrization (to be defined shortly) of an irreducible Markov chain cannot be smaller than the corresponding commute-times in the original chain.

Recently  Gaudilli\`ere and Landim \cite{nonrev} extended much of the classic potential theory to the non-reversible setup and derived several extremal characterizations for the capacity (see \eqref{e:cap} for a definition) between two disjoint sets. In particular, they showed  \cite[Lemma 2.5]{nonrev} that the capacity between two disjoint sets of an irreducible Markov chain having a stationary measure is at least as large as the corresponding capacity w.r.t.\ the additive symmetrization of the chain. In fact, they defined the capacity directly in probabilistic terms (see \eqref{e:cap} below) and established several useful properties for it. See \S\ref{s:cap} for further details. We utilize their framework and exploit their results in order to extend the result of \cite{traces} to the non-reversible setup.

\medskip

A result of similar spirit to that of \cite{traces} was proved by Morris \cite{morris}, who
showed that the components of the wired uniform spanning forest are a.s.\
recurrent. For an a.s.\ recurrence theorem for distributional limits
of finite  planar graphs, see \cite{BS}. In \cite{traces} it was conjectured that the trace of a transient branching (simple) random walk is recurrent for the same branching random walk. A positive answer was given by Benjamini and M\"uller \cite{BRW} in the case that the underlying graph is a Cayley graph (or more generally, a unimodular random graph). 

\subsection{Organization of this work}
In \S\ref{s:Markov} we present some relevant background concerning Markov chains and capacities. In \S\ref{s:overview} we present an overview of the proof of Theorem \ref{thm:1} and reduce it to the validity of Proposition \ref{p:main}. In \S\ref{s:aux} we construct an auxiliary chain and explain its connection with Proposition \ref{p:main}.  In \S\ref{s:proof} we conclude the proof of Theorem \ref{thm:1} by proving Proposition \ref{p:main}. In \S\ref{s:proofof4} and \ref{s:proofof3} we prove Theorems \ref{thm:4} and \ref{thm:3}, respectively. We conclude with some comments concerning the construction from the proof of Theorem \ref{thm:3}, and with a certain relaxation of the assumption that the chain admits a stationary measure.
\section{Some definitions related to Markov chains}
\label{s:Markov}

Consider an irreducible Markov chain  $\X=(X_k)_{k=0}^{\infty}$  on a countable (or finite) state space $V$ with a stationary  measure $\pi$ and transition probabilities given by $P$. We say that $P$ is \emph{reversible} if $\pi(x)P(x,y)=\pi(y)P(y,x)$ for all $x,y \in V$. Observe that a weighted nearest-neighbor random walk is reversible w.r.t.~the measure $\pi$ given by $\pi(v):=c_v$ and hence $\pi$ is stationary for the walk. Conversely, every reversible chain can be presented as a network with edge weights $c_{x,y}=\pi(x)P(x,y)$.
 The \textbf{time-reversal} of $\X$ (w.r.t.\ $\pi$), denoted by $\X^*=(X_k^*)_{k=0}^{\infty}$, is a Markov chain whose transition probabilities are given by $P^*(x,y):=\pi(y)P(y,x)/\pi(x)$. It also has $\pi$ as a stationary measure. Its \textbf{additive symmetrization} $\X^{\sss}=(X_k^{\sss})_{k=0}^{\infty}$ (w.r.t.\ $\pi$) is a reversible Markov chain whose transition probabilities are given by $S:= \half (P+P^*)$. The corresponding edge weights are given by
\begin{equation}
\label{e:cs}
c_{\sss}(x,y):= \half [\pi(x)P(x,y)+\pi(y)P(y,x)]=\half \pi(x)[P(x,y)+P^*(x,y)].
\end{equation}

Note that when the chain is recurrent, $\pi$ is unique up to a constant factor, and so $P_*$ and $S$ are uniquely defined. This may fail when the chain is transient! A transient Markov chain may have two different stationary measures which are not constant multiples of one another.  Consider a random walk on $\Z$ with $P(i,i+1)=p=1-P(i,i-1)$ for all $i \in \Z$, with $p \in ( 1/2,1) $. Note that $P$ is reversible w.r.t.\ $\pi(i)=(p/(1-p))^i $ and  also has the counting measure $\mu$ on $\Z$ as a stationary measure. The additive symmetrization w.r.t.\ $\pi$ is again $P$ while w.r.t.\ $\mu$ is SRW on $\Z$. We see that the additive symmetrization of a transient Markov chain can be recurrent (and that this may depend on the choice of the stationary measure w.r.t.\ which the additive symmetrization is taken). Conversely, Theorem \ref{thm:AG} below asserts that the additive symmetrization of a recurrent Markov chain is always recurrent.

Let $H=(V,E)$ be a graph and $c=(c(e))_{e \in E}$ some edge weights. The \emph{restriction} of the network $(H,c)$ to $A \subset V$, denoted by $(A,c \upharpoonleft A )$, is a network on $(A,E(A))$, where $E(A):=\{\{a,a' \}:a,a' \in A \} $ in which the edge weight of each $e \in E(A)$ equals $c(e)$. 

The \emph{hitting time} of a set $A \subset V$ is $T_A:=\inf \{t \ge 0 :X_t \in A \}$. Similarly, let $T_A^+:=\inf \{t \ge 1 : X_t \in A \} $. When $A=\{x\}$ is a singleton, we write $T_x$ and $T_x^+$ instead of $T_{\{x\}}$ and $T_{\{x\}}^+$. We denote by  $\Pr_a$  the law of the entire chain, started from state $a $.

\medskip

 The $\ell_2$ norm of $f \in \R^V  $ is given by $\|f \|_2:=\sqrt{\langle f,f \rangle_{\pi}} $, where   $\langle f,g \rangle_{\pi}:= \sum_{v \in V}\pi(v) f(v) g(v)$. The space of $\ell_2$ functions is given by $\mathcal{H}:=L_2(V,\pi)=\{f :\R^V:\|f \|_2< \infty \}$.  Then $P$ (and similarly, also $S$ and  $P^*$)  defines a linear operator on $\mathcal{H} $ via $Pf(x):= \sum_{y \in V} P(x,y)f(y)=\mathbb{E}_x[f(X_1)]$. Note that $P^*$ is the dual of $P$ and that $S$ is self-adjoint. The \emph{Dirichlet form} $\mathcal{E}_P(\cdot,\cdot):\mathcal{H}^2 \to \mathbb{R}$ corresponding to $P$    is  $\mathcal{E}_{P}(f,g):=\langle (I-P)f,g \rangle_{\pi}$ (where $I$ is the identity operator). Note that for all $f \in \mathcal{H}$ \begin{equation}
\label{e:Dirichlet}
\mathcal{E}_{P}(f,f)=\mathcal{E}_{P^*}(f,f)=\mathcal{E}_{S}(f,f)=\half  \sum_{x,y}\pi(x)S(x,y)(f(x)-f(y))^2. \end{equation}
Even when $f \notin \mathcal{H} $ we define $\mathcal{E}_{P}(f,f):= \half  \sum_{x,y}\pi(x)P(x,y)(f(x)-f(y))^2 $. With this definition  \eqref{e:Dirichlet} remains true. The quantity $\mathcal{E}_{P}(f,f) $ is called the \textbf{Dirichlet energy} of $f$.
\subsection{Capacity, effective-resistance and voltage}
\label{s:cap}
The \textbf{voltage} function corresponding to the set $A$ (resp.\ to $(A,B)$) is $v_A(x):=\Pr_x[T_A< \infty ] $ (resp.\ $v_{A,B}(x):=\Pr_x[T_A< T_B ]$). When $A=\{y\}$ we write $v_y $ rather than $v_{\{y\}} $.  When $y=o$ we simply write $v$ instead of $v_o$. We denote the voltages corresponding to the time-reversal and additive symmetrization by $v^*$ and $v^{\mathrm{s}}$, respectively.

Let $G=(V,E)$ be a graph. A \textbf{flow} $\theta$ between $A$ and infinity (resp.\ $B$) is a real-valued antisymmetric function on directed edges (i.e.\ $\theta(x,y)=-\theta(y,x)$ and  $\theta(x,y)=0 $ if $\{x,y\} \notin E$)  such that $\sum_{y}\theta(x,y)=0 $ for all $x \notin A$ (resp.\ $x \notin A \cup B$). The \emph{strength} of $\theta$ is defined as $\sum_{a \in A,y \in V }\theta(a,y) $. We say that a flow is a \emph{unit flow} if it has strength 1.

Under reversibility one defines the \textbf{effective-resistance} between a finite set $A$ and infinity (resp.\ $B$) $\mathcal{R}_{A \leftrightarrow \infty } $ (resp.\ $\mathcal{R}_{A \leftrightarrow B }$) as the minimum energy $ \half \sum_{x \neq y}\frac{\theta(x,y)^{2}}{c(x,y)} $ of a unit flow $\theta$ from $A$ to infinity (resp.\ $B$). The flow attaining the minimum is called the unit \emph{current flow} and is given by $i_{A}(x,y):=\frac{c(x,y)}{\con (A)}(v_{A}(x)-v_{A}(y)) $ (resp.\ $i_{A,B}(x,y):=\frac{c(x,y)}{\con (A,B)}(v_{A,B}(x)-v_{A,B}(y)) $), where $\con (A)$ and $\con (A,B)$ are normalizing constants, ensuring that $i_{A}$ and $i_{A,B}$ have strength 1, given by \eqref{e:cap} below (e.g.\ \cite[Ch.\ 2]{lyons}) and  $c(x,y)=\pi(x)P(x,y)$. The reciprocal of $\mathcal{R}_{A \leftrightarrow \infty }$ (resp.\ $\mathcal{R}_{A \leftrightarrow B }$) is called the \textbf{effective-conductance} between $A$ and infinity (resp.\ $B$). Even without reversibility, one defines
the \textbf{capacity} of a set $A$ (resp.~between $A$ and $B$) w.r.t.\ $\pi$ (a stationary measure) as
\begin{equation}
\label{e:cap}
\begin{split}
&\mathrm{Cap}(A):= \sup_{C \subseteq A:\, C \text{ is finite}} \con (C), \\& \text{where for a finite set }C \text{ we define} \quad \con (C):=\sum_{a \in C}\pi(a) \Pr_a[T_{C}^{+}=\infty],  \\ & \mathrm{Cap}(A,B):= \sum_{a \in A}\pi(a) \Pr_a[T_{A}^{+}>T_{B}].
\end{split}
\end{equation}
Under reversibility (for the standard choice $\pi(v):=c_{v}$) the capacity is simply the effective-conductance (e.g.\ \cite[Ch.\ 2]{lyons}). Note that it need not be the case that for an infinite $A$ we have that $\con (A):=\sum_{a \in A}\pi(a) \Pr_a[T_{A}^{+}=\infty]$ (e.g.\ consider SRW on $\Z^3$ and $A=\{(x_1,x_2,x_3) \in \Z^3:x_3=0 \}$). Conversely, in the definition of $\mathrm{Cap}(A,B)$ we assume neither $A$ nor $B$ to be finite. 

It is clear from the definition of the capacity that
an irreducible Markov chain on a countable state space admitting a stationary measure is transient iff $\mathrm{Cap}(x):=\mathrm{Cap}(\{x\})>0$ for some (equivalently, for all) $x \in V$. The following claim will be used repeatedly in what comes. Multiplying all of the edge weights of a network by a constant factor of $M $ has no effect on the walk, and so one can readily see from \eqref{s:cap} that this operation changes $\con (\bullet ) $ by a multiplicative factor of $M$. In particular, if $(G,c)$ and $(G,c')$ are such that $c'(e) \le Mc(e)$  for all $e$ and the network $(G,c)$ is recurrent, then the network $(G,c)$ is also recurrent. Indeed after multiplying all of the edge weights of $c$ by a factor of $M$ we get new weights $\tilde c :=Mc$ with  $c'(e) \le \tilde c(e)$ for all $e$  and so we can apply Rayleigh's monotonicity principle.
Under reversibility,  the \emph{Dirichlet principle} asserts that for every finite set $A$ and all $B$ we have that
\begin{equation}
\label{e:DP}
\begin{split}
& \mathrm{Cap}(A)= \inf_{f \text{ of finite support} :\, f \upharpoonright A \equiv 1 } \mathcal{E}_{P}(f,f) \quad \text{and} \\ &  \mathrm{Cap}(A,B)= \inf_{f:  \text{ the  support of $f$ is finite and contained in }B^{c}, \, f \upharpoonright A \equiv 1\,  } \mathcal{E}_{P}(f,f),
\end{split}
\end{equation}
where $f \upharpoonright A$ is the restriction of $f$ to $A$, $B^{c}:=V \setminus B $ and $\mathcal{E}_{P}(f,f)$ is as in \eqref{e:Dirichlet}. When considering a reversible Markov chain, if we want to emphasize the identity of the edge weights $c$ w.r.t.\ which the capacity is taken we write $\mathrm{Cap}(A;c) $ and $\mathrm{Cap}(A,B;c)$. We denote the capacity of a set $A$ in the restriction of the network to some set $D$ by $\mathrm{Cap}(A;c \upharpoonleft D ) $. 

Using the notation from Theorem \ref{thm:1}, as in \cite{traces} the following consequence of the Dirichlet principle (which we explain below) will be crucial in what comes:
\begin{equation}
\label{e:concave}
\forall \, A \subset V, o \in V \qquad  \mathbb{E}_o[\mathrm{Cap}(A;N)] \le \mathrm{Cap}(A; \mathbb{E}_o [N]) .
\end{equation}
Above $N$ is defined via the Markov chain which need not be reversible. Regardless of this, setting the edge weights to be $N$ or $\mathbb{E}_o [N]$ yields a network. Indeed,
\[2\mathbb{E}_o[\mathrm{Cap}(A;N)] =\mathbb{E}_o[\inf_{f \text{ of finite support} :\, f \upharpoonright A \equiv 1}\sum_{x,y \in V}N(x,y)(f(x)-f(y))^{2}] \] \[ \le \inf_{f \text{ of finite support} :\, f \upharpoonright A \equiv 1} \mathbb{E}_o[\sum_{x,y \in V}N(x,y)(f(x)-f(y))^{2}]=2 \mathrm{Cap}(A; \mathbb{E}_o [N]).   \]

We denote the corresponding capacities w.r.t.~the time-reversal and the additive symmetrization by $\mathrm{Cap}_* $ and $\mathrm{Cap}_{\mathrm{s}} $, respectively. Gaudilli\`ere and Landim \cite[Lemmata 2.2, 2.3 and 2.5]{nonrev} established the following useful properties of the capacity: For every $A \subseteq A',B \subseteq B' \subseteq V$ and every nested sequence of finite sets $V_1 \subseteq V_2 \subseteq \cdots$ which exhausts $V$ (i.e.\ $V=\cup_{n \in \N }V_n$) we have that
\begin{equation}
\label{e:cap2}
\mathrm{Cap}(A)=\mathrm{Cap}_{*}(A) \quad \text{and} \quad \mathrm{Cap}(A,B)= \mathrm{Cap}(B,A)=\mathrm{Cap}_*(A,B) .
\end{equation}
\begin{equation}
\label{e:cap3}
\mathrm{Cap}(A) \le \mathrm{Cap}(A') \quad \text{and} \quad \mathrm{Cap}(A,B) \le \mathrm{Cap}(A',B') .
\end{equation}
\begin{equation}
\label{e:cap4}
\begin{split}
& \mathrm{Cap}(A) =\lim_{n \to \infty} \lim_{m \to \infty} \mathrm{Cap}(A \cap V_n ,V_m^c) \quad \text{and} \\ & \mathrm{Cap}(A,B) = \lim_{n \to \infty}  \lim_{m \to \infty} \mathrm{Cap}(A \cap V_n ,B \cup V_m^c ).
 \end{split}
\end{equation}
\begin{equation}
\label{e:cap5}
\mathrm{Cap}_{\mathrm{s}}(A) \le \mathrm{Cap}(A) \quad \text{and} \quad \mathrm{Cap}_{\mathrm{s}}(A,B) \le \mathrm{Cap}(A,B).
\end{equation}
Note that by \eqref{e:cap3} we have that $\mathrm{Cap}(A) =\lim_{n \to \infty} \mathrm{Cap}(A \cap V_n)$ (i.e.\ the supremum in the definition of  $\mathrm{Cap}(A)$ can be replaced by a limit along an arbitrary exhausting sequence).
\begin{remark}
To be precise, the relations $\mathrm{Cap}(A)=\mathrm{Cap}_{*}(A)$ and  $\mathrm{Cap}(A) \le \mathrm{Cap}(A')$ does not appear explicitly in \cite{nonrev}. They follow easily from the other relations above, which do appear in \cite[Lemmata 2.2, 2.3 and 2.5]{nonrev}. To see this, first observe that it suffices to consider the case that $A$ and $A'$ are finite. Then $\mathrm{Cap}(A)=\lim_{m \to \infty} \mathrm{Cap}(A,V_m^c)=\lim_{m \to \infty} \mathrm{Cap}_{*}(A,V_m^c)=\mathrm{Cap}_{*}(A) $ for a finite $A$. Similarly, if $A \subseteq A'$ are finite then  $\mathrm{Cap}(A)=\lim_{m \to \infty} \mathrm{Cap}(A,V_m^c) \le \lim_{m \to \infty} \mathrm{Cap}(A',V_m^c)=\mathrm{Cap}(A') $.
\end{remark}
\begin{atheorem}[\cite{nonrev} Lemma 5.1]
\label{thm:AG}
Consider a recurrent irreducible Markov chain  $\mathbf{X}$  on a countable state space. Then its additive symmetrization is also recurrent.
\end{atheorem}
As demonstrated earlier, the additive symmetrization of a transient chain w.r.t a certain stationary measure $\pi$ can be recurrent (and this may depend on the choice of $\pi$).
\begin{remark}
Since the expected number of returns to some fixed starting point is the same for $\X $ and its time-reversal $\X^*$, it follows that $\X$ is recurrent iff $\X^*$ is recurrent. This of course follows also from \eqref{e:cap2}.
\end{remark}

The following proposition is well-known in the reversible setup. Although the proof of the general case is essentially identical, we include it for the sake of completeness.
\begin{proposition}
\label{p:Denergyofv}
For every finite set $A$
\[\mathcal{E}_{S}(v_{A},v_A) \le \mathrm{Cap}(A). \]
\end{proposition}
\begin{proof}
We first note that for every $C,D \subset V$ such that $D^{c} $ is finite we have that\footnote{For the first equality in \eqref{e:vcd} note that (1) $ (I-P)v_{C,D} \equiv 0 $ on $F:=(C \cup D)^c$ since $v_{C,D}$ is harmonic on $F$, (2) that $v_{C,D} \equiv 0 $ on $D$ by definition and (3) that $v_{C,D} \upharpoonright C \equiv 1 $, while $(I-P)v_{C,D}(x)=\Pr_x[T_{C}^+>T_{D}] $ for all $x \in C$.} \cite[(2.6)]{nonrev} \begin{equation}
\label{e:vcd}
\mathrm{Cap}(C,D)=\langle (I-P)v_{C,D},v_{C,D} \rangle_{\pi}=\mathcal{E}_{P}(v_{C,D},v_{C,D}). \end{equation}
Let $A \subset V_1 \subset V_2 \subset \cdots $ be a nested sequence of finite sets which exhausts $V$. For every $n$ let $v_n:=v_{A,V \setminus V_n}$. Clearly $\lim_{n \to \infty}v_n(x)=v_A(x)$ for all $x \in V$. Then $\theta_n(x,y):=(v_n(x)-v_n(y))^2 \to \theta(x,y):=(v_A(x)-v_A(y))^{2} $ as $n \to \infty$ for all $x,y$.  By Fatou's Lemma. \[\mathcal{E}_{S}(v_{A},v_A) =\half \sum_{x,y} \pi(x)S(x,y)\liminf_{n \to \infty} \theta_n(x,y) \le \liminf_{n \to \infty}  \sfrac{1}{2} \sum_{x,y} \pi(x)S(x,y)\theta_n(x,y) \] \[ =\liminf_{n \to \infty}\mathcal{E}_{S}(v_{n},v_n)=  \liminf_{n \to \infty}\mathrm{Cap}(A,V \setminus V_n)=\con (A),  \]
where the penultimate equality uses \eqref{e:vcd} and the last equality uses \eqref{e:cap4}.
\end{proof}
\begin{remark}
Under reversibility, for every finite set $A$  one has that $\mathcal{E}_{S}(v_{A},v_A) = \mathrm{Cap}(A) $.
\end{remark}
The \emph{Green function} $\mathcal{G}:V^{2} \to \R_+ $ is given by $\mathcal{G}(x,y):=\sum_{n \in \Z_+}P^n(x,y) $. We have that
\begin{equation}
\label{e:green}
\begin{split}
& \mathcal{G}(x,y)= \frac{\pi(y)}{\pi(x)} \sum_{n \in \Z_+}(P^{*})^n(y,x)=\frac{\pi(y)}{\pi(x)} \mathcal{G}_{*}(y,x) \\ &=\frac{\pi(y)v_{x}^*(y)}{\pi(x)}\mathcal{G}_{*}(x,x)=\frac{\pi(y)v_{x}^*(y)}{\pi(x)}\mathcal{G}(x,x)
\end{split}
\end{equation}
(where $\mathcal{G}_{*}$ is the Green function of the time-reversal). Let
\begin{equation}
\label{e:alpha}
\alpha =\frac{1}{\mathrm{Cap}(o)} =\mathcal{G}(o,o)/ \pi (o)=\mathcal{G}_{*}(o,o)/ \pi (o).
\end{equation}
Observe that by \eqref{e:green}-\eqref{e:alpha}  (using the notation from \eqref{e:cs})
\begin{equation}
\label{e:N}
\begin{split}
& \mathbb{E}_o[N(x,y)] := \mathcal{G}(o,x)P(x,y)+\mathcal{G}(o,y)P(y,x) \\ & =\alpha [ \pi(x)P(x,y)v^*(x)+ \pi(y)P(y,x)v^*(y) ] \le 2 \alpha c_{\sss}(x,y) \max \{ v^*(x),v^*(y) \}.
\end{split}
\end{equation}
In what comes \eqref{e:N} will be crucial in comparing the networks $(G(\X),\mathbb{E}_o[N])$ and $(G(\X), c_{\sss})$. For further details on this point see \eqref{e:motivation} and the paragraph following it.

\section{An overview of the proof of Theorem \ref{thm:1}}
\label{s:overview}
Following \cite{traces} the following sets shall play a major role in what comes
\begin{equation}
\label{Adel}
\begin{split}
& A_{\delta}:=\{a:v(a) \ge \delta \}, \quad V_{\delta}:=\{a:v(a) < \delta \},
\\ & A_{\delta}^*:=\{a:v^{*}(a) \ge \delta \}, \quad V_{\delta}^*:=\{a:v^{*}(a) < \delta \}.
\end{split}
\end{equation}
When certain claims are true by symmetry for both $A_{\delta} $ and $A_{\delta}^* $, we sometimes phrase them in terms of $A_{\delta}$. However as \eqref{e:N} suggests, the sets  $A_{\delta}^* $ are more relevant for us than the sets  $A_{\delta}$.
In fact, for our purposes we need not consider  $A_{\delta} $ and $V_{\delta}$ at all. We find it counter-intuitive at first sight that in order to study the trace of $\X$ one has to investigate its behavior w.r.t.\ $v^*$ and the sets $A_{\delta}^* $ (rather than w.r.t.\ $v$ and $A_{\delta}$).  
\begin{lemma}
\label{lem:connected}
Let $\X$ be a transient irreducible Markov chain on a countable state space. Then in the above notation we have that for every $u \in V$ there exists a path $u_0=u,u_1,\ldots, u_{\ell}=o$ in $G(\X)$ such that for all $1 \le i \le \ell$ we have that $P(u_{i-1},u_{i})>0$ and $v(u) \le v(u_{i})$. Thus the restriction of the graph $G(\X)$ to $A_{\delta} $ is connected. The same holds for $A_{\delta}^* $ with $P^*$ in the role of $P$.
\end{lemma}
\begin{proof}
Assume towards a contradiction that no such path exists. Let $B$ be the collection of all $y$'s such that there exists a path $u_0=u,u_1,\ldots, u_{\ell}=y $ (for some $\ell \in \N$) in $G(\X)$ such that for all $0 \le i < \ell$ we have that $P(u_i,u_{i+1})>0$ and $v(u) \le v(u_{i+1})$. Observe that if $w \notin B$ and $P(b,w)>0$ for some $b \in B $ then $v(w)<v(u)$ (as otherwise we would have that $w \in B$). Consider the case that $X_0=u$.  Since by assumption $o \notin B $  we have that $v(X_{n \wedge T_{B^c} }) $ is a bounded non-negative martingale (if $o \in B$ it would have been only a super-martingale). By \eqref{e:i.o.} $\lim_{n \to \infty} \mathbb{E}_u[v(X_{n})]=\Pr_u[o \text{ is visited infinitely often}]=0 $ and hence $T_{B^c}<\infty $ a.s. (as $\Pr_u[T_{B^c}>t]v(u) \le \Pr_u[T_{B^c}>t] \mathbb{E}_u[v(X_{t}) \mid T_{B^c}>t] \le\mathbb{E}_u[v(X_{t})] $).  By optional stopping we get that $v(u)=\mathbb{E}_u[v(X_{T_{B^c} })]$. However this fails, as $v(X_{T_{B^c} })<v(u)$. A contradiction!
\end{proof}
 
 The main ingredient in the proof of Theorem \ref{thm:1} is the following proposition.
\begin{proposition}
\label{p:main}
Let $\mathbf{X}$ be a transient irreducible Markov chain admitting a stationary measure on a countable state space $V$. Assume that $\mathbf{X}^{\sss}$ is also transient.   Let $o \in V$. Then
\begin{equation}
\label{e:main}
\forall \,  \delta>0, \qquad  \mathrm{Cap}(A_{\delta}^*;\mathbb{E}_o[N]) \le 2.
\end{equation}
\end{proposition}
We now present the proof of Theorem \ref{thm:1}, assuming \eqref{e:main}.

\medskip

\emph{Proof of Theorem \ref{thm:1}:} Let $G=G(\X)$ be as in \eqref{e:G(X)}. By Remark 1 from \S\ref{s:comments} (which we now briefly recall)  it suffices to show that  $(G, N )$ is a.s.\ recurrent (which is equivalent to   $(\mathrm{PATH}, N )$ being recurrent): As SRW on  $\mathrm{PATH}$ is the random walk on  $(\mathrm{PATH}, \mathbf{1} )$,  Rayleigh's monotonicity principle (e.g.~\cite[Ch.\ 2]{lyons} or \cite{doyle}), together with the fact that $N \ge 1$ on the edges of $\mathrm{PATH} $, implies that if  $(\mathrm{PATH},N)$ is recurrent then so is SRW on  $\mathrm{PATH}$. By \eqref{e:concave} it is in fact sufficient to show that   $(G, \mathbb{E}_o[N] )$ is recurrent. Indeed if    $(G, \mathbb{E}_o[N] )$ is recurrent, then by \eqref{e:concave} $\mathbb{E}_o[\mathrm{Cap}(o;N)]=0 $. As $\mathrm{Cap}(o;N) \ge 0 $ we get that $\mathrm{Cap}(o;N)=0 $ a.s..  

We now show $(G, \mathbb{E}_o[N] )$ is recurrent. If $\X^{\sss} $ is recurrent, then  by \eqref{e:N}  $(G, \mathbb{E}_o[N] )$ is also recurrent, as its edge weights are (edge-wise) larger than those of $\X^{\sss}$ by at most a bounded factor of $2 \alpha$.
If $\X^{\sss} $ is transient then   by \eqref{e:main}  $\lim_{\delta \to 0} \mathrm{Cap}(A_{\delta}^*;\mathbb{E}_o[N]) \le 2 $.  If the network  $(G,\mathbb{E}_o[N])$ was transient, then
by Lemma \ref{lem:2.3} below and the fact that $A_{\delta}^* \nearrow V $ as $\delta \to 0$, in conjunction with \eqref{e:cap3}, we would have  $\lim_{\delta \to 0} \mathrm{Cap}(A_{\delta}^*;\mathbb{E}_o[N]) = \infty $, a contradiction!\qed
\begin{lemma}[\cite{traces} Lemma 2.3]
\label{lem:2.3}
Let $G=(V,E)$ be a graph. If $(G,c)$ is a transient network then for all $M>0$ there exists a finite set $A$ such that $\mathrm{Cap}(A;c) \ge M$.
\end{lemma}
The proof below is taken from \cite{traces} and is given below for the sake of completeness.
\begin{proof}
Let $\theta$ be a unit flow of finite energy from a vertex $o$ to $\infty$. Since
$\theta$ has finite energy, there is some finite set $\{o\} \subseteq A \subset V$ such that the energy of $\theta$ on the
edges with some endpoint not in $A$ is less than $1/M$ (i.e., $\sum_{x,y \in V :x \in A^c,y \sim x }\theta^2(x,y)/c(x,y)<1/M $). Hence $\theta'(x,y):=\theta(x,y)\Ind{\{x,y\} \nsubseteq A} $ is a unit flow from $A$ to $\infty$ whose energy is at most $1/M$. That is, the effective-resistance from $A$ to infinity is at most $1/M$. 
\end{proof}
\section{Subdividing edges}
\label{s:aux}
Recall that $A_{\delta}^*=\{a:v^{*}(a) \ge \delta \}$ and that $V_{\delta}^*:=V \setminus A_{\delta}^* $. Following \cite{traces} the following sets shall play a major role in our analysis: 
\begin{equation}
\label{Adel}
\begin{split}
 & W_{\delta}^*:=\{a \in A_{\delta}^*:P(a,V_{\delta}^*)>0 \}=``\pd_{\mathrm{v}}^{\mathrm{int}} A_{\delta}^* =\text{internal vertex-boundary of }A_{\delta}^* ", \quad  \\ & U_{\delta}^*:=\{b\in V_{\delta}^* :P^{*}(b,W_{\delta}^*)>0 \}=\{b \in V_{\delta}^*:\sum_{a \in A_{\delta}^* } P(a,b)>0 \} \\ & =``\pd_{\mathrm{v}}^{\mathrm{ext}} A_{\delta}^* =\text{external vertex-boundary of }A_{\delta}^*".
\end{split}
\end{equation}
We used quotes above as $W_{\delta}^*$ (resp.\ $U_{\delta}^* $) is the internal (resp.\ external) vertex-boundary of $A_{\delta}^* $ only in some directed sense. For technical reasons, we consider also the sets
\begin{equation*}
\begin{split}
 & \widehat W_{\delta}^*:=\{a \in W_{\delta}^*: v^{*}(a)> \delta  \} \quad \text{and}
\\ & \widehat U_{\delta}^* :=\{b \in U_{\delta}^*: P^{*}(b, \widehat W_{\delta}^*)>0 \}=\{b \in U_{\delta}^* :\sum_{a \in \widehat W_{\delta}^* } P(a,b)>0 \} .
\end{split}
\end{equation*}
Similarly to \cite{traces}, a crucial ingredient in the argument (behind Proposition \ref{p:main}) is that for every fixed $\delta \in (0,1) $, instead of $\X$ we can consider an auxiliary Markov chain\footnote{For each $\delta$ we construct a separate auxiliary  chain. For the sake of notational convenience we suppress the dependence of $\delta$ from some of the notation.} $\widehat \X :=(\widehat X_n)_{n=0}^{\infty} $ in which the edges between $ \widehat W_{\delta}^* $ and $ \widehat U_{\delta}^* $ are subdivided\footnote{The reason we subdivide only edges between $\widehat W_{\delta}^*$  $\widehat U_{\delta}^*$ and not all edges between $ W_{\delta}^*$ and  $ U_{\delta}^* $ is that it is crucial that the terms in the second line of   \eqref{e:sub1} will lie strictly in $(0,1)$.}, so that $\widehat \X$ can access the collection of states whose voltage w.r.t.\ the time-reversal of $\widehat \X$ is less than $\delta$ only by first passing through states at which this voltage is precisely $\delta$. Under reversibility our construction agrees with the one from \cite{traces}. In general, our construction is more involved as instead of specifying the edge weights and applying an elementary network reduction, we have to specify  the  transition probabilities of the auxiliary chain in terms of that of the original chain. Unfortunately, we cannot apply a network reduction to the auxiliary chain in order to argue that it is in some sense equivalent to the original chain. We also have to check separately that certain properties hold for $\widehat \X$ while other hold for its time-reversal $\widehat \X^*:=(\widehat X_n^*)_{n=0}^{\infty}$.

\subsection{Motivation for subdividing edges}
\label{s:motivation}

To motivate what comes, before diving into the details of the construction of the auxiliary chain, we first motivate it and then describe several properties which we want the construction to satisfy. Recall that we denote the capacity of a set $A$ in the restriction of the network $(G,c)$ to some set $D$ by $\mathrm{Cap}(A;c \upharpoonleft D ) $.    
It follows from the analysis in \S\ref{s:proof} that
\begin{equation}
\label{e:motivation}
\begin{split}
& \con (A_{\delta}^*;\mathbb{E}_o[N]) =\con ( W_{\delta}^*;\mathbb{E}_o[N] \upharpoonleft  W_{\delta}^* \cup V_{\delta}^* ).
\\ & \con ( W_{\delta}^*;c_{\sss} \upharpoonleft  W_{\delta}^* \cup V_{\delta}^* )=\con (A_{\delta}^*;c_{\sss}) \le \sfrac{1}{ \alpha \delta} .
\end{split}
\end{equation} 
Now, if it was the case that $v^* \upharpoonleft  W_{\delta}^* \equiv \delta $, then by \eqref{e:N} $\mathbb{E}_o[N(x,y)] \le 2 \alpha \delta c_{\sss}(x,y) $ for all $x,y \in  W_{\delta}^* \cup V_{\delta}^*  $. It would have then followed that \[\con (A_{\delta}^*;\mathbb{E}_o[N]) =\con ( W_{\delta}^*;\mathbb{E}_o[N] \upharpoonleft  W_{\delta}^* \cup V_{\delta}^* ) \le 2 \alpha \delta \con ( W_{\delta}^*;c_{\sss} \upharpoonleft  W_{\delta}^* \cup V_{\delta}^* ) \le 2,\] which is precisely the assertion of Proposition \ref{p:main}. The auxiliary chain we construct below allows us to overcome the fact that in practice we might not have that $v^* \upharpoonleft  W_{\delta}^* \equiv \delta $.

Equation \eqref{e:motivation} will not be proven explicitly and is only meant to serve as motivation for the construction of the auxiliary chain and for the analysis in Section \ref{s:proof}.  We end this section by giving some intuition behind the \eqref{e:motivation}. The equalities in \eqref{e:motivation} follow from the following  general principle for reversible chains (see Lemma \ref{lem:rest}): if $B$ is a connected set and  for every $b \in B$ the walk (w.r.t.\ edge weights $c$) started from $b$ a.s.\ visits $B$ only finitely many times then $\con (B;c)=\con (\pd B;c \upharpoonleft B^c \cup \pd B) $, where $\pd B:=\{b  \in B:P(b,B^c)>0\} $ is the internal vertex-boundary of $B$. The above condition on $B$  for $B=A_{\delta}^*$ for the networks induced by the edge weights  $\mathbb{E}_o[N] $ and $c_{\sss} $ is verified in Lemma \ref{lem:v}. We note that above we have been somewhat imprecise, as the boundary $\pd A_{\delta}^*$ (w.r.t.\ the edge weights  $\mathbb{E}_o[N] $ and $c_{\sss} $)   may be larger than $W_{\delta}^*$ (it is equal to the set $W$ from Corollary \ref{cor:5.4}). In section \S\ref{s:proof} we are of course more precise. In fact, over there we have to work also with the auxiliary chain we define below, for which the condition corresponding to  $v^* \upharpoonleft  W_{\delta}^* \equiv \delta $ holds by construction. Because of that, a corresponding statement to $\con (A_{\delta}^*;\mathbb{E}_o[N]) \le 2$ will indeed hold for the auxiliary chain (while the above derivation  was conditional on   $v^* \upharpoonleft  W_{\delta}^* \equiv \delta $). Fortunately, the auxiliary chain is constructed to be intimately related to the original chain, in a manner which allow us to translate back the aforementioned corresponding statement to the desired one for the original chain.   

The inequality $\con (A_{\delta}^*;c_{\sss}) \le \sfrac{1}{ \alpha \delta} $ follows from Lemma \ref{lem:v2}, which asserts that  $\con_* (A_{\delta}^*) \le \sfrac{1}{ \alpha \delta} $ (where $\con_* $ denotes capacity w.r.t.\ the time-reversal), along with the fact that by \eqref{e:cap5} we have that $ \con (A_{\delta}^*;c_{\sss}) =\con_{\sss} (A_{\delta}^*) \le \con_* (A_{\delta}^*) $. We note that the proof of Lemma \ref{lem:v2} is different than that of the corresponding result in \cite{traces} which relies on reversibility.       

\subsection{Description of the auxiliary chain and some of its properties}

The state space of the auxiliary Markov chain is
\begin{equation}
\label{e:JZ}
\begin{split}
& \widehat V := V \cup Z, \quad \text{where} \quad Z:=\{z_{a,b}: (a,b) \in J \} \text{ and} \\ & J:=\{(a,b) \in \widehat W_{\delta}^*  \times \widehat U_{\delta}^* :P(a,b)>0  \}.
\end{split}
\end{equation}

We employ the convention that for all $a,b$ such that $(a,b) \in J$ \[z_{a,b}=z_{b,a}.\] 
We emphasize that $z_{a,b}$ is for each $(a,b) \in J$ a state of the  auxiliary Markov chain.  The  auxiliary Markov chain is obtained from the original chain by subdividing the edge $\{a,b\}$ into two edges $\{a,z_{a,b}\}$ and $\{z_{a,b},b\}$ connected in series in a way we describe below. Namely, we shall define the transition matrix $\widehat P$ of the auxiliary chain such that for all   $x \in V \setminus ( \widehat W_{\delta}^* \cup \widehat U_{\delta}^* )$ we have that $\widehat P(x,y)=P(x,y)$ for all $y \in V$ and $\widehat P(x,z)=0$ for all $z \in Z$. We shall define $\widehat P$ so that in particular for all $(a,b) \in J$ we have that  $\widehat P(z_{a,b},a)+ \widehat P(z_{a,b},b)=1$ and $\widehat P(x,z_{a,b})=0$ for all $x \in \{a,b\}^c$ (this justifies the usage of the term ``subdivision" above, and the description of  $\{a,z_{a,b}\}$ and $\{z_{a,b},b\}$ as being ``connected in series").  Hence we only need to specify that transition probabilities away from states in $Z \cup \widehat W_{\delta}^*  \cup \widehat U_{\delta}^* $ which shall be chosen carefully to ensure that the non-lazy version of the induced chain of $\widehat \X $ on $V$ is $\X$, which loosely speaking means that when we view $\widehat \X $ at times it is in $V$ we obtain a realization of $\X$ (this description is slightly imprecise, see property (d) below for a precise statement).

Recall that the time-reversal of $\widehat{\mathbf{X}}^* $ is denoted by  $\widehat \X^* :=(\widehat X_n^*)_{n=0}^{\infty}$. To be precise, below we actually first construct a certain auxiliary chain $\widehat \X^*  $, find a stationary measure for it $\hat \pi$, and then define $\widehat \X$ as the time-reversal of $\widehat \X^* $ w.r.t.\ $\hat  \pi $ (this of course ensures that $\widehat \X^* $ is the time-reversal of $\widehat \X $ w.r.t.\ $\hat \pi$, and so there is no abuse of notation).

 Voltages and capacities w.r.t.~$\widehat \X $ (resp.~$\widehat \X^* $) shall be denoted by $\hat v $ and $\widehat \con$ (resp.\ $\hat v^* $ and $\widehat{ \con_{*}} $).  Its  transition kernel and a stationary measure for it (defined in (f) below)  shall be denoted by $\widehat P $ and  $\hat \pi $. We denote the corresponding probability and expectation for $\widehat \X$ by $\widehat \Pr $ and $\mathbb{ \widehat E}$. Finally, let \[ \widehat N(x,y)=\widehat N(y,x):=| \{n:  \{\widehat X_n, \widehat X_{n+1}\}= \{x,y  \} \} |  \] be the number of (undirected) crossings of the edge $\{x,y \}$ by the chain $\widehat \X $.

The (desirable) condition  $v^* \upharpoonleft  W_{\delta}^* \equiv \delta $  from \S\ref{s:motivation} can be phrased as the condition that  $v^*(x)=\delta $ for all $x$ such that $v^*(x) \ge \delta $ and $P(x,y)>0$ for some $y$ such that $v^*(y)<\delta $. While this may fail, the key property that the auxiliary chain will posses is that   $\hat v^*(x)=\delta $ for all $x \in \widehat V $ such that $\hat v^*(x) \ge \delta $ and $\widehat P(x,y)>0$ for some $y \in \widehat V$ such that $\hat v^*(y)<\delta $. This is in fact the reason we define the auxiliary chain.
 
\begin{remark}
We stress the fact that the sets $A_{\delta}^*,\widehat W_{\delta}^{*} $ and $\widehat U_{\delta}^{*} $ below are (as above) all defined w.r.t.\ $\X$, not w.r.t.\ $\widehat \X$.
\end{remark}

We assume, without loss of generality that $P(x,x)=0=P^*(x,x)$ for all $x \in V $.\footnote{Otherwise we may replace $P$ by $P'(x,y):=\frac{P(x,y)\Ind{x \neq y}}{1-P(x,x)}$.} The chain $\widehat \X$ is constructed to satisfy the following properties\footnote{Below (a)-(e) are the more important properties.} (some of the properties below were already mentioned above, but are repeated below to facilitate ease of reference):
\begin{itemize}
\item[(a)] $\hat v ^* (u)=v^*(u) $ for all $u \in V$, where $\hat v^*(u):=\hat v^*_o(u)=\widehat \Pr_u^{*}[T_o< \infty]$ and $\widehat \Pr^* $ is the law of $\widehat \X^* $.
\item[(b)] For all $z \in Z$ we have that $ \hat v ^* (z)=\delta $.  Moreover, for every $(a,b) \in J  $ we have that  \[\widehat P(z_{a,b},a)+ \widehat P(z_{a,b},b)=1=\widehat P^{*}(z_{a,b},a)+\widehat P^{*}(z_{a,b},b). \]
Equivalently, $\widehat P(z_{a,b},a)+ \widehat P(z_{a,b},b)=1$ and $\widehat P(x,z_{a,b})=0$ for $x \in \{a,b\}^c$ and the same holds for $\widehat {P}^*$.
 Hence by (a) it must be the case that for all $(a,b) \in J$ we have that
\begin{equation}
\label{e:sub1}
\begin{split}
& \widehat P^{*}(z_{a,b},a)v^*(a)+\widehat P^{*}(z_{a,b},b)v^*(b)= \delta \quad \text{and so} \quad \\ & \quad \widehat P^{*}(z_{a,b},a)=\frac{ \delta-v^*(b)}{v^*(a)-v^*(b)}>0 \quad \text{and} \quad \widehat P^{*}(z_{a,b},b)=\frac{ v^*(a)-\delta}{v^*(a)-v^*(b)} >0.
\end{split}
\end{equation}
\item[(c)] Let $G(\X)=(V,E)$ be the graph supporting the transitions of $\X$ as in \eqref{e:G(X)}. Then $G(\widehat \X):=(\widehat V,\widehat E )$ is the graph supporting the transitions of $\widehat \X$, where \[ \widehat E:= ( E \setminus \{ \{a,b \}:(a,b) \in  J\}) \cup \{ \{a,z_{a,b} \},\{z_{a,b},b \}:(a,b) \in J \} \] and $\widehat V $ is as in \eqref{e:JZ}. In words, $G(\widehat \X)$ is obtained from $G(\X)$ by subdividing   each edge $\{a,b \}$ with $(a,b) \in J$ into two edges $\{a,z_{a,b}\} $ and $\{z_{a,b},b \}$ connected in series.
\item[(d)] The non-lazy version of the induced chain of $\widehat \X $ on $V$ is $\X$. More precisely, 
\begin{equation}
\label{e:d1}
\forall \, x , y \in V,\; \widehat \Pr_x[\widehat X_T=y]=P(x,y), \quad \text{where} \quad T:=\inf \{t>0: \widehat X_t \in V \setminus \{ \widehat X_0  \} \}. \end{equation}
Moreover, for all $(a,b) \in J $ and all  $(b,a) \in  J$ we have that
\begin{equation}
\label{e:d2}
\widehat \Pr_a[\widehat X_T=b, \widehat X_{T-1}=z_{a,b}]=P(a,b). \end{equation}
Consequently, if we view the chain $\widehat \X $ only when it visits $V $ and then omit from it all lazy steps (i.e.\ consecutive visits to the same state) then the obtained chain is  $\X$ (here we are using the assumption that $P(x,x)=0$ for all $x \in V $).

More precisely, let $t_0=0 $ and inductively set $t_{i+1}:=\inf \{t>t_{i}:\widehat X_{t} \in V \}$. Let $Y_{i}:=\widehat X_{t_i} $ for all $i$. Then $\mathbf{Y}:=(Y_i)_{i=0}^{\infty}$ is the chain $\widehat \X$ viewed only when it visits $V$. We now transform $\mathbf{Y}$ into a Markov chain $ \mathbf{\widehat Y}:=(\widehat Y_i)_{i=0}^{\infty}$ whose holding probabilities are 0, by omitting repetitions (i.e.\ lazy steps), setting  $\widehat Y_i:=Y_{s_i}$ for all $i$, where  $s_0=0$ and inductively $s_{i+1}:=\inf \{t>s_i:Y_{t} \neq Y_{s_i} \}$. Then by \eqref{e:d1}-\eqref{e:d2} $ \mathbf{\widehat Y}$ is distributed as $\X$. \item[(e)] Consequently, for all $\{ x,y \} \in E $ such that $(x,y),(y,x) \notin J $ we have that
 \[ \mathbb{\widehat E}_o[ \widehat N(x,y)]=\mathbb{E}_o [ N(x,y)] \]
(where $\mathbb{\widehat E}_o$ denotes expectation w.r.t.\ $\widehat \X $), while for all $(a,b) \in  J$
\[ \mathbb{\widehat E}_o[ \widehat N(z_{a,b},b)]  \ge \mathbb{\widehat E}_o[ | \{  i \ge 0 : \{\widehat X_{t_i},\widehat X_{t_{i+1}}\}= \{a,b \}, \widehat X_{t_{i+1}-1}=z_{a,b}   \}  |  ]= \mathbb{E}_o [ N(a,b)]. \]
\item[(f)] A stationary measure $\hat \pi$ of $\widehat \X$ is given by: $\hat \pi (u):=\pi(u) $ for  $u \in V \setminus ( \widehat W_{\delta}^* \cup \widehat U_{\delta}^* ) $, 
\[\forall \, u \in \widehat W_{\delta}^* \cup \widehat U_{\delta}^* , \qquad \hat \pi (u) := \frac{\pi(u)}{1-\sum_{z \in Z } \widehat P(u,z)\widehat P(z,u) }, \]
and $\hat \pi(z_{a,b}) = \hat \pi(a)\widehat P(a,z_{a,b})+ \hat \pi(b)\widehat P(b,z_{a,b})  $ for all $z_{a,b}\in Z$.
\item[(g)] For all $(x,y) \in (V \setminus ( \widehat W_{\delta}^* \cup \widehat U_{\delta}^* )) \times  V $ we have that
 $\widehat P(x,y) =P(x,y)$. \item[(h)] For all $(a,b) \in  J $ and all  $(b,a) \in  J$
 \[\hat \pi(a) \widehat P(a,z_{a,b})\widehat P(z_{a,b},b)= \pi(a) P(a,b). \]
Finally, for every $(x,y) \in V^2  $ such that $(x,y) \notin J$ and $(y,x) \notin J$
 \[\hat \pi(x) \widehat P(x,y) = \pi(x) P(x,y). \]
\end{itemize}
\begin{remark}
\label{rem:hatZ}
Observe that the definitions of $\widehat W_{\delta}^*,\widehat U_{\delta}^*$  and $J$ (together with property (c)) are taken to ensure that the chain $\widehat \X $ cannot move from a state $x$ with $\hat v^*(x)>\delta $ directly to a state $y$ with $\hat v^*(y)<\delta$ (the same however need not apply to $\widehat \X^*$). In other words,
\[A_{\delta}^* \setminus (W_{\delta}^* \setminus \widehat W_{\delta}^* ) = \{u \in \widehat V: \hat v^*(u) > \delta  \},  \]
\[ V_{\delta}^* = \{u \in \widehat V: \hat v^*(u) < \delta  \}, \]
\[\widehat Z:=Z \cup W_{\delta}^* \setminus \widehat W_{\delta}^* =\{u \in \widehat V: \hat v^*(u) = \delta  \}  \supseteq \{u \in \widehat V: \hat v^*(u)  \ge \delta \text{ and }\widehat P(u,V_{\delta}^* )>0 \}  . \]
\end{remark}
\subsection{Construction of the auxiliary chain}
 We now prove the existence of such $\widehat \X$ satisfying properties (a)-(h).  In fact, because we already have \eqref{e:sub1} and also because property (a) is defined in terms of $\widehat \X^* $, it is more convenient to first construct $\widehat \X^*$ (and then define $\widehat \X$ as its time reversal w.r.t.\ $\hat \pi$). We will construct $\widehat \X^* $ to satisfy (a)-(c) and some analogous properties to (d)-(h) from which (d)-(h) (for $ \widehat \X $) will follow. Most importantly, we define $\widehat \X^* $ so that properties (a)-(c) hold and so that (similarly to property (d)) for $T^{*}:=\inf \{t>0: \widehat X_t^* \in V \setminus \{ \widehat X_0^*  \} \} $ we have that for all $x , y \in V$  
\begin{equation}
\label{e:d3}
\widehat \Pr_x^*[\widehat X_{T^*}^*=y]=P^{*}(x,y), \end{equation}
while for all $(a,b) \in J $ and all  $(b,a) \in  J$ we further have that
\begin{equation}
\label{e:d4}
\widehat \Pr_a^{*}[\widehat X_{T^*}^*=b, \widehat X_{T^*-1}^*=z_{a,b}]=P^*(a,b). \end{equation}
Before proving \eqref{e:d3}-\eqref{e:d4} we note that they immediately imply property (a).

 We start by determining the transition probabilities of $\widehat \X^* $. Let
\[D:= \{w \in \widehat W_{\delta}^* : \nexists u \in V \text{ s.t. } (w,u)\in J \text{ and }P^{*}(w,u)>0 \}. \]
 For all $x \in(V \setminus ( \widehat W_{\delta}^* \cup \widehat U_{\delta}^* )) \cup D$  we set
\begin{equation}
\label{e:J0}
\widehat P^*(x,y) :=\begin{cases}P^{*}(x,y) & y \in V \\
0 & y \in Z \\
\end{cases}.
\end{equation}
Recall that by \eqref{e:sub1} $\widehat P^{*}(z_{a,b},a)=\frac{ \delta-v^*(b)}{v^*(a)-v^*(b)}>0 $ and $\widehat P^{*}(z_{a,b},b)=\frac{ v^*(a)-\delta}{v^*(a)-v^*(b)}>0 $ for all $(a,b) \in J$. In accordance with property (c) we set $\widehat P^{*}(a,b)=0=\widehat P^{*}(b,a) $ for all $a,b$ such that  $(a,b) \in J$. 

We now define the transition probabilities of $\widehat P^*(a,\bullet) $ for $a \in \widehat W_{\delta}^* \setminus D$ and then in a similar manner also  for $b \in \widehat U_{\delta}^*$.
Fix some $a \in \widehat W_{\delta}^* \setminus D$. For all $y \in V $ such that $P^*(a,y)=0 $ and $(a,y) \notin J $ (resp.\  $(a,y) \in J $) we define $\widehat P^*(a,y)=0$ (resp.\ $\widehat P^*(a,z_{a,y})=0$). By the definition of $D$  there exists some   $b$ such that $(a,b) \in J$ and $P^*(a,b)>0 $. Fix one such $b=b(a)$.   For all $y \in V  $ such that $P^*(a,y)>0 $ and $(a,y) \notin J $ we define $\widehat P^*(a,y)$ and $\widehat P^*(a,z_{a,b})$ so that\footnote{Note that there are two undetermined probabilities in \eqref{e:J3}, $\widehat P^*(a,y)$ and $\widehat P^*(a,z_{a,b}) $, however, crucially, for every fixed $a$ we always use the same $b$ in the r.h.s.\ of \eqref{e:J3} and \eqref{e:J4}.}
\begin{equation}
\label{e:J3}
\widehat P^*(a,y)/P^{*}(a,y) =\frac{\widehat P^*(a,z_{a,b})\widehat P^*(z_{a,b},b)}{P^{*}(a,b) }  \end{equation}
and for all  $y \in V $ such that $P^*(a,y)>0 $ and $(a,y) \in J$ we define $\widehat P^*(a,z_{a,y}) $ so that
\begin{equation}
\label{e:J4}
\widehat P^*(a,z_{a,y})\widehat P^*(z_{a,y},y)/P^{*}(a,y)=\frac{\widehat P^*(a,z_{a,b})\widehat P^*(z_{a,b},b)}{P^{*}(a,b) }
\end{equation}
 (recall that $\widehat P^*(z_{a,y},y)$ and $P^*(z_{a,b},b) $ have already been determined). Using \eqref{e:sub1} and normalizing so that \[\sum_{y \in V :(a,y) \in J}\widehat P^*(a,z_{a,y}) +\sum_{y \in V  :(a,y) \notin J} \widehat P^*(a,y)=1\] the transition probabilities of $\widehat P^*(a,\bullet) $ are determined. More explicitly, we get that
\[1=\frac{\widehat P^*(a,z_{a,b})\widehat P^*(z_{a,b},b)}{P^{*}(a,b) } L(a), \]
where $L(a):=  \sum_{y :(a,y) \notin J}P^{*}(a,y) + \sum_{y:(a,y) \in J}\frac{P^{*}(a,y)}{\widehat P^*(z_{a,y},y)}$. Observe that by \eqref{e:sub1} $L(a) \le \frac{v^*(a)}{v^*(a)-\delta} < \infty $. Hence we can set $\widehat P^*(a,z_{a,b}):=\frac{P^{*}(a,b) }{ L(a) \widehat P^*(z_{a,b},b)} $ (where $\widehat P^*(z_{a,b},b)$ is already defined in \eqref{e:sub1}) and then define $\widehat P^*(a,\bullet)$ using \eqref{e:J3}-\eqref{e:J4}. 

Similarly, for every $b \in \widehat U_{\delta}^* $ we fix some $a=a(b) \in \widehat W_{\delta}^* $ such that $P^*(b,a)>0$ (note that by construction there exists such $a$).  For all $y \in V $ such that $P^*(b,y)=0 $ and $(y,b) \notin J $  we define $\widehat P^*(b,y)=0$. For all $y \in V $ such that $P^*(b,y)>0 $ and $(y,b) \notin J $ (i.e.\ $P^*(b,y)>0$ and $y \notin \widehat W_{\delta}^* $) we define $\widehat P^*(b,y)$ so that
\begin{equation}
\label{e:J1}
\widehat P^*(b,y) /P^{*}(b,y) =\frac{\widehat P^*(b,z_{a,b})\widehat P^*(z_{a,b},a)}{P^{*}(b,a) }  \end{equation}
and for all  $y $ such that $(y,b) \in J$ (and so by construction  $P^*(b,y)>0 $) we define $\widehat P^*(b,z_{y,b})$ and $\widehat P^*(b,z_{a,b})$ so that
\begin{equation}
\label{e:J2}
\widehat P^*(b,z_{y,b})\widehat P^*(z_{y,b},y)/P^{*}(b,y)=\frac{\widehat P^*(b,z_{a,b}) \widehat P^*(z_{a,b},a)}{P^{*}(b,a) }, \end{equation}
As before, using \eqref{e:sub1} and normalizing so that $\sum_{y \in \widehat W_{\delta}^* :P^*(b,y)>0 }\widehat P^*(b,z_{y,b}) +\sum_{y \in V \setminus \widehat W_{\delta}^* }\widehat P^*(b,y)=1$ the transition probabilities away from $b$ are determined. 

We now verify \eqref{e:d3}-\eqref{e:d4}. If $x \in (V \setminus (\widehat W_{\delta}^* \cup \widehat U_{\delta}^* )) \cup D $ then \eqref{e:d3} holds trivially by \eqref{e:J0}. 
Let $a \in \widehat W_{\delta}^* \setminus D $ and let $b=b(a)$ be as above.
Observe that \eqref{e:J3} in conjunction with \eqref{e:J4} imply that  $\frac{\widehat \Pr_a^*[\widehat X_{T^*}^*=y]}{\widehat \Pr_a^*[\widehat X_{T^*}^*=b,\widehat X_{T^*-1}^*=z_{a,b}]}=\frac{P^{*}(a,y)}{P^{*}(a,b)}  $ for all  $y \in V  $ such that $(a,y) \notin J $ and  that $\frac{\widehat \Pr_a^*[\widehat X_{T^*}^*=y,\widehat X_{T^*-1}^*=z_{a,y}]}{\widehat\Pr_a^*[\widehat X_{T^*}^*=b,\widehat X_{T^*-1}^*=z_{a,b}]}=\frac{P^{*}(a,y)}{P^{*}(a,b)}  $ for all  $y \in V $ such that $(a,y) \in J$. Summing over all $y$ we get that $\frac{1}{\widehat \Pr_a^*[\widehat X_{T^*}^*=b,\widehat X_{T^*-1}^*=z_{a,b}]}=\frac{1}{P^{*}(a,b)} $ and so  \eqref{e:d3}-\eqref{e:d4} indeed hold for $x \in \widehat W_{\delta}^* $ for all $y$. Similarly, the case that $x \in \widehat U_{\delta}^*$ follows from  \eqref{e:J1} in conjunction with \eqref{e:J2} in the roles of  \eqref{e:J3} and \eqref{e:J4} from the analysis of the previous case.   

For every  $u \in V $ let  \[\beta_u:=\sum_{z  \in Z } \widehat P^{*}(u,z)\widehat P^{*}(z,u)=\sum_{z \in Z } \widehat P(u,z)\widehat P(z,u).\]
Note that $\beta_u=0 $ if $u  \notin ( \widehat W_{\delta}^* \setminus D) \cup \widehat U_{\delta}^* $. Let   $u \in ( \widehat W_{\delta}^* \setminus D) \cup \widehat U_{\delta}^* $. Observe that \[\widehat \Pr_u^*[T^*<T_{u}^+]=1-\beta_u.\] By \eqref{e:d3}-\eqref{e:d4} and the ``craps principle" \cite[p.\ 210]{pitman},    for all $u' \in V $ such that $(u,u') \notin J$ and also $(u',u) \notin J$ we have that 
\begin{equation}
\label{e:piu3}
P^{*}(u,u')=\widehat \Pr_u^*[\widehat X_{T^*}^*=u'] =\widehat P^*(u,u')/(1-\beta_u), \end{equation}
while for all $u' \in V $ such that $(u,u') \in J$ or $(u',u) \in J$ we  have that
\begin{equation}
\label{e:piu4}
P^{*}(u,u')=\widehat \Pr_u^*[\widehat X_{T^*}^*=u',\widehat X_{T^*-1}^*=z_{u,u'}]= \widehat P^*(u,z_{u,u'})\widehat P^*(z_{u,u'},u')/(1-\beta_u). 
\end{equation}
%
It is not hard to verify that a stationary measure $\hat \pi$ for $\widehat \X^* $ is given by:  
\begin{equation}
\label{e:piu}
\forall \, u \in V \setminus ( \widehat W_{\delta}^* \cup \widehat U_{\delta}^* ), \qquad \hat \pi (u)=\pi(u), 
\end{equation} 
\begin{equation}
\label{e:piu2}
\forall \, u \in \widehat W_{\delta}^* \cup \widehat U_{\delta}^*, \qquad \hat \pi (u) = \pi(u)/(1-\beta_u) , 
\end{equation}
\begin{equation}
\label{e:piu6}
\forall \, (a,b) \text{ such that }  (a,b) \in J \text{ or }  (b,a) \in  J,\;  \hat \pi(z_{a,b})= \hat \pi(b)\widehat   P^*(b,z_{a,b})+\hat \pi(a)\widehat   P^*(a,z_{a,b}),
\end{equation}
from which (using \eqref{e:piu4} and \eqref{e:piu2} for \eqref{e:piu7}) we see that
\begin{equation}
\label{e:piu7}
\hat \pi(a) \widehat P^{*}(a,z_{a,b})\widehat P^{*}(z_{a,b},b)= \pi(a) P^{*}(a,b),
\end{equation}
 and also (using \eqref{e:piu3} and \eqref{e:piu2}) that if $(x,y) \in V^2$ is such that $(x,y) \notin J$ and $(y,x) \notin J$ then
\begin{equation}
\label{e:piu5}
\widehat \pi(x) \widehat P^*(x,y)=\pi(x)P^*(x,y).
\end{equation}

Indeed, while \eqref{e:piu}-\eqref{e:piu5} can be derived via  probabilistic considerations, they can also be verified via a direct calculation (treating \eqref{e:piu}-\eqref{e:piu6} as a guess and then verifying that $\hat \pi \widehat P^*= \hat \pi $) which we now carry out. Indeed, if we define $\hat \pi$ via \eqref{e:piu}-\eqref{e:piu6} then it is straightforward to check that \eqref{e:piu7}-\eqref{e:piu5} hold. For $u \in (V \setminus ( \widehat W_{\delta}^* \cup \widehat U_{\delta}^* ))\cup D  $ by  \eqref{e:piu} and \eqref{e:piu5} we have 
\[\sum_{x\in \hat V }\hat \pi (x) \widehat P^*(x,u)=\sum_{x\in  V }\hat \pi (x) \widehat P^*(x,u)= \sum_{x\in V } \pi (x)  P^*(x,u)  =\pi(u), \]
while for  $a \in   \widehat W_{\delta}^* \setminus D  $   by \eqref{e:piu5},  \eqref{e:piu6} and \eqref{e:piu7}  (used, respectively, in the first, second and third equations below) we have that  
\[\sum_{x\in \hat V }\hat \pi (x) \widehat P^*(x,a)= \sum_{x\in V \setminus \widehat U_{\delta}^* } \pi (x)  P^*(x,a) +\sum_{z\in Z} \hat \pi (z) \widehat P^*(z,a) \]
\[=\sum_{x\in V \setminus \widehat U_{\delta}^* } \pi (x)  P^*(x,a)+ \sum_{b \in \widehat U_{\delta}^* } \hat \pi (b) \widehat   P^*(b,z_{a,b})\widehat    P^*(z_{a,b},a)+ \hat \pi (a) \sum_{z\in Z}  \widehat  P^*(a,z) \widehat  P^*(z,a) \]
 \[=\sum_{x \in V} \pi(x) P^*(x,a)+\hat \pi (a) \beta_a=\pi(a)+\hat \pi (a) \beta_a= \hat \pi (a).  \]
The case that $u \in \widehat U_{\delta}^*$ is similar and is omitted. 

Using the fact that $(\widehat P^*)^*=\widehat P $ we can now determine the transition probabilities of $\widehat P $. Properties (b)-(c) hold by construction of $\widehat P^*$. By \eqref{e:piu} and \eqref{e:piu5}  we have that property (g) holds, as for all $(x,y) \in ( (V \setminus ( \widehat W_{\delta}^* \cup \widehat U_{\delta}^* )) \cup D) \times  V $ we have that
 \[\widehat P(x,y) =\frac{\hat \pi(y) \widehat P^*(y,x)}{\hat \pi(x)}=\frac{\hat \pi(y) \widehat P^*(y,x)}{ \pi(x)}=\frac{ \pi(y)  P^*(y,x)}{ \pi(x)}=P(x,y).\] Property (h) follows from \eqref{e:piu7}-\eqref{e:piu5}. Property (f) follows from \eqref{e:piu}-\eqref{e:piu2}. We now verify that property (d) holds. We prove \eqref{e:d1}, leaving \eqref{e:d2} as an exercise. If $(x,y) \notin J$ and $(y,x) \notin J$ then by the crabs principle and property (h) (used  in the first and last equalities, respectively)
\[\widehat \Pr_x[\widehat X_T=y]=\frac{\widehat P(x,y)}{1-\beta_x}=\frac{\hat \pi(x) \widehat P(x,y)}{\hat \pi(x) (1-\beta_x)} =\frac{\hat \pi(x) \widehat P(x,y)}{ \pi(x)}=P(x,y), \]
while if $(x,y) \in J$ or $(y,x) \in J$ then the same reasoning yields that
\[\widehat \Pr_x[\widehat X_T=y]=\frac{\widehat P(x,z_{x,y})\widehat P(z_{x,y},y)}{1-\beta_x} =\frac{ \hat \pi (x) \widehat P(x,z_{x,y})\widehat P(z_{x,y},y)}{\pi(x) } =P(x,y). \]
Property (e) is a consequence of (d). This concludes the proof. \qed
\section{Proof of Proposition \ref{p:main}}
\label{s:proof}
The following lemma is (essentially) borrowed from \cite{traces}. The proofs of some of the parts are identical to the corresponding proofs in \cite{traces}, while others are different due to the lack of reversibility assumption.
\begin{lemma}
\label{lem:v}
 Let $\mathbf{X}$ be a transient irreducible Markov chain admitting a stationary measure on a countable state space $V$. Let $o \in V$. Then
\begin{itemize}
\item[(i)]
$\inf_{x \in V}v(x)=0$ (where $v=v_{o}$).
\item[(ii)]
The restriction of the network $(G(\X),c_{\sss} )$ to $A_{\delta}$ (where $c_{\sss}(\cdot,\cdot)$ are the edge weights of the additive symmetrization $\X^{\sss}$, defined in \eqref{e:cs}) is recurrent for all $\delta \in (0,1)$.
\item[(iii)] The restriction of the network $(G(\X),\mathbb{E}_o[N] )$ to $A_{\delta}^*$ is recurrent for all $\delta \in (0,1)$.

\item[(iv)]
$v(X_n) \to 0 $ almost surely.
\item[(v)] Let  $\X^{\mathbb{E}[N]}=(X_n^{\mathbb{E}[N]})_{n=0}^{\infty} $ be the random walk on the network $(G(\X),\mathbb{E}_o[N])$. If it is transient then $v^{*}(X_n^{\mathbb{E}[N]}) \to 0 $ a.s..
\end{itemize}
\end{lemma}
\begin{proof}
By transience $0=\Pr_o[o \text{ is visited infinitely often}]=\Pr_o[\cap_{n \ge 0}\cup_{m \ge n}\{ X_m=o \} ]$. Thus  
\begin{equation}
\label{e:i.o.}
\begin{split}
0 & = \Pr_o[o \text{ is visited infinitely often}] = \lim_{n \to \infty} \Pr_o[X_m=o \text{ for some }m \ge n ] \\ & = \lim_{n \to \infty}\mathbb{E}_o[v(X_n)] \ge \inf_{x \in V}v(x).
\end{split}
\end{equation}

Thus $\inf_{x \in V}v(x)=0$. Fix some $\delta \in (0,1)$. 
We now prove Part (ii).  The  Doob-transformed Markov chain corresponding to the conditioning of returning to $o$ has transition probabilities $ Q (x,y):=\frac{ v(y)}{v(x)}P(x,y) $ for $x \neq o$ (while $ Q(o,y):=P(o,y)v(y)/\sum_x P(o,x)v(x) $ for all $y$) and stationary measure $\tilde \pi(x):=\pi(x) v(x)$. This chain is obviously recurrent. By Theorem \ref{thm:AG} its additive symmetrization, whose edge weights are given by 
\begin{equation*}
\begin{split}
\tilde c_{\sss}(x,y)&=\half [\tilde \pi(x) Q(x,y)+\tilde \pi(y)Q(y,x)]=\half [\pi(x)P(x,y)v(y)+\pi(y)P(y,x)v(x)] \\ &  \ge c_{\sss}(x,y) \min \{ v(x),v(y)\} ,
\end{split}
\end{equation*}
is also recurrent.  By Rayleigh's monotonicity principle, its restriction to $A_{\delta} $ is recurrent as well. Since for $x,y \in A_{\delta}$ we have that $  c_{\sss}(x,y) \le \delta^{-1} \tilde c_{\sss}(x,y) $, it follows that the restriction of $\X^{\sss}$ to $A_{\delta} $ is also recurrent. This concludes the proof of Part (ii).

We now prove Part (iii). By Part (ii) and symmetry\footnote{Throughout we use the symmetry between claims regarding $\X$ and the corresponding claims for $\X^*$.} the restriction of $(G(\X),c_{\sss} )$ to $A_{\delta}^*$ is also recurrent. By \eqref{e:N}  the edge weights  of the restriction of the network $(G(\X),\mathbb{E}_o[N] )$ to $A_{\delta}^*$ are (edge-wise) within a constant factor of those of the restriction of $(G(\X),c_{\sss} )$ to $A_{\delta}^*$. Hence the former must also be recurrent.

We now prove Part (iv). Observe that $v(X_n) $ is a non-negative super-martingale and hence converges to a limit a.s.\ and in $L_1$. Since $v(X_n)  \in [0,1]  $ for all $n$, its limit also belongs to $[0,1]$. By \eqref{e:i.o.}   $\mathbb{E}_o[v(X_n)]\to 0 $. Hence the limit of $v(X_n)$ must be 0 a.s.. 

We now prove Part (v). Assume that $\X^{\mathbb{E}[N]}$ is transient. By Proposition \ref{p:Denergyofv} (and symmetry) $v^*$ has finite Dirichlet energy w.r.t.\ the network $(V,E,c_{\sss})$. By \eqref{e:N}  $v^*$ has finite Dirichlet energy also w.r.t.\ the network $(G(\X),\mathbb{E}_{o}[N])$. By  Ancona, Lyons and Peres \cite{ALP} $\lim_{n \to \infty} v^{*}(X_n^{\mathbb{E}[N]})$ exists a.s.. By Part (iii) the limit must equal 0 a.s.. 

We now explain this in more detail. Assume towards a contradiction that for some $\eps>0$, $\mathbb{P}_o[\lim_{n \to \infty} v^{*}(X_n^{\mathbb{E}[N]})>\eps]>0 $.  Then there is some $x$ so that  $\mathbb{P}_x[\inf_{n \ge 0 } v^{*}(X_n^{\mathbb{E}[N]}) \ge \eps/2]>0 $. Started from $x$ we can couple $X_n^{\mathbb{E}[N]} $ with its restriction to $A_{\eps/2}^* $ until the first time the former leaves this set. It follows that with positive probability, in this coupling the two chains are equal to each other for all times. However, by Part (iii) the latter a.s.\ visits $o$ infinitely often. We get that if $\X^{\mathbb{E}[N]} $ is transient and  $\mathbb{P}_o[\lim_{n \to \infty} v^{*}(X_n^{\mathbb{E}[N]}) = 0 ]<1 $ then with positive probability it would visit $o$ infinitely often, a contradiction!    
\end{proof}
The following lemma will later be used with $B=A_{\delta}^* $ and also on the auxiliary chain from \S\ref{s:aux} with  $B=A_{\delta}^* \cup Z $.
\begin{lemma}
\label{lem:rest}
Let $H=(V,E)$ be a graph. Let $(H,c)$ be a transient network with $c(e)>0$ for all $e \in E$. Let $B \subset V $ be such that (1) the induced subgraph $H \upharpoonleft B$ is connected and (2) the network random walk on
$H$ starting at any $b \in B $ visits $B$ only finitely many times a.s.. Let $\pd_{\vv}^{\mathrm{int}}B:=\{b \in B: c(b,x)>0 \text{ for some }x \notin B \}$ and $D:=\pd_{\vv}^{\mathrm{int}}B \cup B^{c} $. Then
\[\mathrm{Cap}(B;c )=\mathrm{Cap}(\pd_{\vv}^{\mathrm{int}}B;c )=\mathrm{Cap}(\pd_{\vv}^{\mathrm{int}}B;c \upharpoonleft D ). \]
\end{lemma}
\begin{proof}
Since $\pd_{\vv}^{\mathrm{int}}B\subseteq B$ by \eqref{e:cap3} we have that $\mathrm{Cap}(B) \ge \mathrm{Cap}(\pd_{\vv}^{\mathrm{int}}B) $. Conversely, fix some finite $A \subseteq B $. Let $\eps >0$. By assumption (2) we can pick a finite $W$ such that $A \cap  \pd_{\vv}^{\mathrm{int}}B \subseteq  W \subseteq \pd_{\vv}^{\mathrm{int}}B$ and for all $a \in A_1:= A \setminus \pd_{\vv}^{\mathrm{int}}B $ we have that $\Pr_a[T_{W}=\infty]<\frac{\eps }{|A| \max_{a \in A}  \pi(a)  }$. Then
\begin{equation*}
\begin{split}
\mathrm{Cap}(A  )& \le \mathrm{Cap}(A \cup W ) \le \sum_{a \in A_1  }\pi(a) \Pr_a[T_{ W}=\infty]+ \sum_{w \in W  }\pi(w) \Pr_w[T_{A \cup W}^{+}=\infty] \\ & \le \eps + \sum_{w \in W  }\pi(w) \Pr_w[T_{A \cup W}^{+}=\infty] \le \eps + \sum_{w \in W }\pi(w) \Pr_w[T_{ W}^{+}=\infty] \\ & = \eps + \con (W) \le \eps + \mathrm{Cap}(  \pd_{\vv}^{\mathrm{int}}B).  \end{split}
 \end{equation*}
Taking the supremum over all finite $A$ and sending $\eps \to 0$ we get that $\mathrm{Cap}(B) \le \mathrm{Cap}(  \pd_{\vv}^{\mathrm{int}}B) $. To conclude the proof  we now show that indeed $\mathrm{Cap}(  \pd_{\vv}^{\mathrm{int}}B;c)= \mathrm{Cap}(  \pd_{\vv}^{\mathrm{int}}B;c \upharpoonleft D )$. By Rayleigh's monotonicity principle $\mathrm{Cap}(  \pd_{\vv}^{\mathrm{int}}B;c) \ge \mathrm{Cap}(  \pd_{\vv}^{\mathrm{int}}B;c \upharpoonleft D )$. Conversely, let $A$ be some finite subset of $  \pd_{\vv}^{\mathrm{int}}B $. Let $\eps >0$.   By Lemma \ref{lem:russ} below there exists some finite $A' \subseteq   \pd_{\vv}^{\mathrm{int}}B $ such that \[ \mathrm{Cap}(A;c )/(1+ \eps ) \le \mathrm{Cap}(A';c\upharpoonleft D) \le \mathrm{Cap}(\pd_{\vv}^{\mathrm{int}}B;c\upharpoonleft D).   \]
Sending $\eps$ to 0 and taking supremum over all finite $A \subseteq \pd_{\vv}^{\mathrm{int}}B$ concludes the proof.
\end{proof}
The following lemma is a minor reformulation of \cite[Lemma 9.23]{lyons}. One difference is that below we take $A$ to be a finite set and not a singleton. However a standard construction in which the vertices of $A$ are identified to a single vertex can be used to reduce the case that $A$ is a finite set  to the case that $A$ is a singleton.
\begin{lemma}[\cite{lyons} Lemma 9.23]
\label{lem:russ}
Let $H=(V,E)$. Let $(H,c)$ be a transient  network.  Let $A$ be a finite subset of $V $ and $B \subset V $ be
such that (1) $A \subseteq  B$,
\begin{itemize}
\item[(2)] the induced subgraph $H \upharpoonleft B$ is connected, and 
\item[(3)] the  random walk on
$(H,c)$ starting at any $b \in B $ visits $B$ only finitely many times a.s..
\end{itemize}
 Let $i$ be the unit current flow on $H$ from
$A$ to $\infty$ and $\eps > 0$. Then there is a unit flow $\theta$ on $H$ from $A$ to $\infty$ such that $\theta(e) \neq 0$ for only
finitely many edges $e$ incident to $B$ and
\[\sum_{e \in E \setminus (B \times B ) }\theta(e)^2/c(e) \le \eps\sum_{e \in E }i(e)^2/c(e)+ \sum_{e \in E \setminus (B \times B ) }i(e)^2/c(e). \]
In particular, the restriction of $\theta$ to $ E \setminus (B \times B )$ is a unit flow in the restriction of the network  $(H,c)$  to $V \setminus B' $, where $B':=\{b \in B:b \text{ is not adjacent to any vertex in }B^{c} \} $, from the finite set $A':=A \cup \{b \in B \setminus  B':\theta(b,b') \neq 0 \text{ for some }b' \in B' \}  $ to infinity, whose energy is at most $ (1+\eps)\sum_{e \in E }i(e)^2/c(e)=(1+\eps)/\con (A;c) $.  
\end{lemma}
Recall the auxiliary chain $\widehat \X $ from \S\ref{s:aux} and the  notation from there. In particular, recall that  $ \widehat Z:= Z  \cup W_{\delta}^* \setminus \widehat W_{\delta}^*  $ as in Remark \ref{rem:hatZ} and that $ \widehat c_{\sss}$ are the edge-weights of $\widehat \X^{\sss} $, the additive symmetrization of $\widehat \X $. Recall that capacities w.r.t.~$\widehat \X $ (respectively, $\widehat \X^* $) are denoted by $\widehat \con$ (respectively, $\widehat{ \con_{*}} $).
\begin{corollary}
\label{cor:5.4}
Let $\delta>0$.  Let $\widetilde W_{\delta}:=\{a \in A_{\delta}^* :P^*(a,V_{\delta}^*)>0=P(a,V_{\delta}^*) \}$, $W:=W_{\delta}^* \cup \widetilde W_{\delta} $ and $\widetilde Z:=\widehat Z \cup \widetilde W_{\delta}   $.  Then
\begin{equation}
\label{e:russ1}
\mathrm{Cap}(A_{\delta}^*;\mathbb{E}_o[N])= \mathrm{Cap}(W;\mathbb{E}_o[N] \upharpoonleft W \cup V_{\delta}^*  ),
\end{equation}
\begin{equation}
\label{e:russ2}
\widehat \con (A_{\delta}^* \cup Z ; \widehat c_{\sss}  )=\widehat \con (\widetilde Z ; \widehat c_{\sss}  )=\widehat \con (\widetilde Z ; \widehat c_{\sss} \upharpoonleft \widetilde Z \cup V_{\delta}^*).  \end{equation}
\end{corollary}
\begin{proof}
If the network $(G(\X),\mathbb{E}_o[N]) $ is recurrent then both sides of \eqref{e:russ1} equal zero. We now consider the case it is transient.
By Part (v) of Lemma \ref{lem:v} $\X^{\mathbb{E}[N]} $ a.s.\ visits $A_{\delta}^*$ only finitely many times. Hence \eqref{e:russ1} follows by Lemma \ref{lem:rest} with $B $ there taken to be $A_{\delta}^*$ (and so $\pd_{\vv}^{\mathrm{int}}B=W $). The connectivity condition from Lemma \ref{lem:rest} is satisfied by Lemma \ref{lem:connected}.

We now prove \eqref{e:russ2}. Again, we may assume that the network $(G(\widehat \X), \widehat c_{\sss}  ) $ is transient, as otherwise all terms in \eqref{e:russ2} equal 0.   
 Similarly, by Part (ii) of Lemma \ref{lem:v}, applied to $\widehat \X$  in conjunction with Remark \ref{rem:hatZ} (which asserts that $A_{\delta}^* \cup  Z$ plays for $\widehat \X$ the same role $A_{\delta}^*$ plays from $\X$) we have that the random walk corresponding to the restriction of the  network $(G(\widehat \X), \widehat c_{\sss}  )$ to $A_{\delta}^* \cup  Z$ is recurrent. Since we assume  the  network $(G(\widehat \X), \widehat c_{\sss}  ) $ is transient, the same reasoning as in Part (v) of Lemma \ref{lem:v} now yields that the random walk corresponding to the network  $(G(\widehat \X), \widehat c_{\sss}  ) $   a.s.\ visits $A_{\delta}^* \cup  Z $ only finitely many times (indeed, by the reasoning from the  proof of Part (v) of Lemma \ref{lem:v}  and using   the notation from there,  $\lim_{n \to \infty} \hat v^*(\widehat{X}_n^{c_{\sss}}) = 0 $, a.s.,  where $(\widehat{X}_n^{c_{\sss}})_{n \in \Z_+} $ is the walk corresponding to the network  $(G(\widehat \X), \widehat c_{\sss}  )$).
Hence \eqref{e:russ2} follows from Lemma \ref{lem:rest} applied to the network $(G(\widehat \X), \widehat c_{\sss}  ) $ with $B $ from Lemma \ref{lem:rest} taken to be $A_{\delta}^* \cup Z  $ (and so $\pd_{\vv}^{\mathrm{int}}B = \widetilde Z $). Again, the connectivity condition from Lemma \ref{lem:rest} is satisfied by Lemma \ref{lem:connected} (as well as the fact that $A_{\delta}^* \cup  Z$ plays for $\widehat \X$ the same role $A_{\delta}^*$ plays from $\X$).  
\end{proof}
The following lemma is also borrowed from \cite{traces}. Again the proof below is different due to the lack of reversibility assumption.
\begin{lemma}
\label{lem:v2}
In the above setup and notation we have that for all $\delta \in (0,1)$
\begin{equation}
\label{e:capadel}
\mathrm{Cap}(A_{\delta}) \le \mathrm{Cap}(o) / \delta .
\end{equation}
\end{lemma}
\begin{proof}
 Fix some $0< \eps < \delta < 1$ and  $ x \in A_{\delta}$. By  Lemma \ref{lem:v} $\Pr_{x}[T_{\{ o\} \cup V_{\eps} } < \infty]=1 $ and so by the definition of $V_{\eps}$ and the Markov property
\begin{equation}
\label{e:last}
\delta 1_{x \neq o} \le \Pr_{x}[T_{o}^+ < \infty] \le\Pr_{x}[T_{o}^+ < T_{V_{\eps}}]+\Pr_{x}[T_{o}^+ < \infty \mid T_{o}^+ > T_{V_{\eps}}] \le \Pr_{x}[T_{o}^+ < T_{V_{\eps}}]+\eps .
\end{equation}
Substituting $x=o$ yields $\mathrm{Cap}(o)=\lim_{\epsilon \to 0}\mathrm{Cap}(o,V_{\eps})=\lim_{\epsilon \to 0}\mathrm{Cap}(V_{\eps},o)$. By the strong Markov property (applied to $T_{A_{\delta}}$) together with \eqref{e:last}   we get that  for every $u \in V_{\epsilon}$ we have 
\[\Pr_{u}[T_{o} <  T_{V_{\eps}}^{+}] \ge \Pr_{u}[T_{A_{\delta}} <  T_{V_{\eps}}^{+}] (\delta - \epsilon).   \]
Multiplying by $\pi(u)$ and then summing over all $u \in V_{\eps}$ we get that \[\mathrm{Cap}(V_{\eps},o) \ge \mathrm{Cap}(V_{\eps},A_{\delta})[\delta - \eps]  .\]
Finally,  by \eqref{e:cap4} 
\[\mathrm{Cap}(o)=\lim_{\epsilon \to 0}\mathrm{Cap}(V_{\eps},o) \ge \lim_{\epsilon \to 0}\mathrm{Cap}(V_{\eps},A_{\delta})[\delta - \eps]=\delta \lim_{\epsilon \to 0}\mathrm{Cap}(A_{\delta},V_{\eps}) \] \[ \ge \delta \sup_{A \text{ finite }: A \subseteq A_{\delta} } \lim_{\epsilon \to 0}\mathrm{Cap}(A,V_{\eps}) =\delta \sup_{A \text{ finite }: A \subseteq A_{\delta} } \mathrm{Cap}(A)= \delta \mathrm{Cap}(A_{\delta}), \]
where the penultimate equality follows from the fact that for all  $0< \eps < \delta <1$ and all finite $A \subseteq A_{\delta}$, by the Markov property (applied to $T_A$ in the penultimate inequality below) we have that for all $a \in A$  \[0 \le \Pr_a[T_{A}^+< \infty]- \Pr_a[T_{A}^+< T_{V_{\epsilon}}] \le \sup_{y \in V_{\epsilon}}\Pr_y[T_{A}<  \infty] \le \delta^{-1}\sup_{y \in V_{\epsilon}}\Pr_y[T_{o}<  \infty] \le \eps /\delta.\]
 Hence $\mathrm{Cap}(A)=\lim_{\epsilon \to 0}\mathrm{Cap}(A,V_{\eps}) $ for all finite $A \subseteq A_{\delta}$.
\end{proof}

\subsection{Proof of Proposition \ref{p:main}}
\emph{Proof:} Fix $\delta \in (0,1)$. Recall the auxiliary chain $\widehat \X $ from \S\ref{s:aux} and the  notation from there. We denote the capacity of a set $A$ in the restriction of the network $(G(\X),c)$ (respectively, $(G(\widehat \X),c )$) to a set $B$ by $\con (A;c \upharpoonleft  B)$ (respectively,  $\widehat  \con (A;c \upharpoonleft  B)$), where $c$ are some arbitrary edge weights. Let $W,\widetilde W_{\delta}, \widetilde Z $ and $\widehat Z$ be as in Corollary \ref{cor:5.4}. We first argue that
\begin{equation}
\label{e:5.6}
\con (A_{\delta}^*;\mathbb{E}_o[N])=\con (W;\mathbb{E}_o[N] \upharpoonleft  W\cup V_{\delta}^* ) \le \widehat \con ( \widetilde Z ;\mathbb{\widehat E}_o[\widehat N] \upharpoonleft \widetilde Z\cup V_{\delta}^* ),   \end{equation}
where the equality is precisely  \eqref{e:russ1}. For the inequality above, modify the restriction of $(G(\widehat \X),\mathbb{\widehat E}_o[\widehat N]  )$ to $\widetilde Z\cup V_{\delta}^*$ as follows:   for every $a \in \widehat W_{\delta}^*$ such that $J(a):=\{b \in V_{\delta}^*:(a,b) \in J  \} \neq \eset$ we identify the set $\{z_{a,b}:b \in J(a) \}$ as a single vertex labeled by $a$. This change cannot change the capacity on the r.h.s.\ of \eqref{e:5.6}, since the states we are identify all belong to $\widetilde Z$. After this is done, the set $\widetilde Z$ is transformed into  the set $W$ and the inequality in  \eqref{e:5.6} becomes a consequence of Rayleigh's monotonicity principle, using property (e) from \S\ref{s:aux}. We now argue using \eqref{e:N} (here applied to $\widehat \X$) and again property (e) from \S\ref{s:aux} that
\begin{equation}
\label{e:5.7caphat}    \widehat \con ( \widetilde Z ;\mathbb{\widehat E}_o[\widehat N] \upharpoonleft \widetilde Z \cup V_{\delta}^*) \le \sfrac{2\delta}{\widehat \con (o)} \widehat \con (\widetilde Z ; \widehat c_{\sss} \upharpoonleft \widetilde Z \cup V_{\delta}^*),  
\end{equation}
where $\widehat c_{\sss} $ are the edge weights of the additive symmetrization of $\widehat \X $. To see this, observe that by \eqref{e:N} and the definition of the set $\widetilde Z $, the edge weights in the restriction of the network    $(G(\widehat \X);\widehat c_{\sss} ) $    to $\widetilde Z \cup V_{\delta}^*$ are (edge-wise) larger than the corresponding edge weights in  the restriction of the network  $(G(\widehat \X);\mathbb{\widehat E}_o[\widehat N]) $ to $\widetilde Z \cup V_{\delta}^*$ by at most a factor of $\frac{2\delta}{\widehat \con (o)} $. Indeed, if $z \in \widehat Z $ and $u \in V_{\delta}^*$ then $\max \{\hat v^*(z),\hat v^*(u) \}=\delta $ (and then the claim follows by \eqref{e:N}), while if $w \in \widetilde W_{\delta}  $ then for all $b \in V_{\delta}^*$ we have that $\widehat P(b,w)=P(b,w)=0  $. Thus by \eqref{e:N} $\mathbb{\widehat E}_o[\widehat N(w,b)]= \mathcal{\widehat G}(o,b)\widehat P(b,w)=\frac{2\widehat c_{\sss}(b,w)\hat v^*(b) }{\widehat \con (o)}<\frac{2\widehat c_{\sss}(b,w)\delta  }{\widehat \con (o)} $, where $ \mathcal{\widehat G}(o,b):=\sum_{n=0}^{\infty}\widehat P^n(o,b) $ is the Green's function of $\widehat \X $.

We now argue that
\begin{equation}
\label{e:new}
\widehat \con (\widetilde Z ; \widehat c_{\sss} \upharpoonleft \widetilde Z  \cup V_{\delta}^*) \le \widehat \con(o)/ \delta .
\end{equation}

By \eqref{e:russ2} $\widehat \con (\widetilde Z ; \widehat c_{\sss} \upharpoonleft \widetilde Z  \cup V_{\delta}^*)=\widehat \con (\widetilde Z  ; \widehat c_{\sss}  )= \widehat \con_{\sss}(\widetilde Z )$. By \eqref{e:cap5}  $ \widehat \con_{\sss}(\widetilde Z ) \le  \widehat \con_{*}(\widetilde Z ) $ and by Lemma \ref{lem:v2} and properties (a)-(b) from \S\ref{s:aux} (namely, that $\hat v^* (z)=\delta $ for all $z \in \widehat Z$ and that $\hat v^* (w)=v^*(w) \ge \delta$ for all $w \in \widetilde W_{\delta}$) we have that $  \widehat \con_{*}(\widetilde Z) \le   \widehat \con_{*}(o)/ \delta=   \widehat \con(o)/ \delta $. Thus
\[\widehat \con (\widetilde Z ; \widehat c_{\sss} \upharpoonleft \widetilde Z  \cup V_{\delta}^*) =\widehat \con_{\sss}(\widetilde Z ) \le  \widehat \con_{*}(\widetilde Z ) \le    \widehat \con(o)/ \delta, \]
 as desired. Finally, combining \eqref{e:5.6}-\eqref{e:new}  yields that
\[\con (A_{\delta}^*;\mathbb{E}_o[N]) \le \sfrac{2\delta}{\widehat \con (o)} \widehat \con (\widetilde Z ; \widehat c_{\sss} \upharpoonleft \widetilde Z \cup V_{\delta}^*) \le  2,    \]
which concludes the proof. \qed

\section{Proof of Theorem \ref{thm:4}}
\label{s:proofof4}
Let $f,g:\N \to \R_{+}$. We write $f \lesssim g $ and $g \gtrsim f $ if there exists $C \ge 1 $ such that  $ f(k) \le C g(k) $ for all $k$. We write $f \asymp g $ if  $f \lesssim g \lesssim  f $.   Recall the definitions of $g_{s}$ and $\log_{*}^{(s)}$ from Theorem \ref{thm:4}. Below we consider birth and death chains on $\Z_+$ (or on $\{0,1,\ldots,n\} $) with holding probability 0 (i.e., $P(x,x)=0$ for all $x$).   Recall that a birth and death chain on $\Z_+$ with weights $(\omega(n,n+1):n \in \Z_+) $ is a reversible Markov chain with $P(n,n+1)=\sfrac{\omega(n,n+1)}{\omega(n,n+1)+\omega(n,n-1)}=1-P(n,n-1)$ (where $\omega(0,-1):=0$) for all $n \in \Z_+$. Throughout the section, we employ the convention that all of the edge weights are symmetric, i.e., $\omega(n,n+1)=\omega(n+1,n) $ for all $n \in \Z_+$ (which is consistent with the general reversible setup).  In this section we prove the following theorem.
\begin{theorem}
\label{thm:ata melagleg alay} Let $s \in \Z_+$. Let $R$ be a birth and death chain with weights $w(k-1,k) \asymp g_{s+2}(k)(\log_{*}^{(s+3)}k)^{2}$. Then there exists some $c(s)>0$ such that with positive probability $R$ crosses each edge $(k,k+1)$ at least $c(s)g_s(k)$
times. 
\end{theorem}
The proof of Theorem \ref{thm:ata melagleg alay} will be carried by induction.  We first have to analyze the case where $w(k-1,k) \asymp k(\log_{*}k)^{2}$ and then the case where  $w(k-1,k) \asymp g_{1}(k)(\log_{*}^{(2)}k)^{2}$.

Recall that $t$ is a cut-time of a Markov chain $(X_n)_{n=0}^{\infty}$ if $\{X_n:  n \le t\}$ and $\{X_n:   n > t\}$ are  disjoint. We say that $ \ell  $ is a \emph{regeneration}
point if the walk visits $\ell$ only once (equivalently, in the birth and death setup, $T_{\ell}$ is a cut-time). 
As noted in \S\ref{s:related}, Benjamini, Gurel-Gurevich and Schramm  \cite{cut3} showed that for every transient Markov chain the expected number of cut-times  is infinite.   The following example, which appeared earlier in \cite{JLP}, a.s.\ has only finitely many cut-times.  This was previously proven in \cite{JLP}. We  present the proof for the sake of completeness. Recall the convention that $w(k-1,k)=w(k,k-1) $.  
\begin{proposition}
\label{p::bd1}
Let $R$ be a birth and death chain with weights $w(k-1,k):=\omega_{k-1} \asymp k\log^{2}k$.
Then a.s.\ $R$ has only finitely many cut-times. Moreover, for every fixed $M>0$ there exists some $p=p(M)>0$ such that with probability at least $p $ for all $k$ the edge $\{k,k+1\}$ is crossed at least $2M+1$ times.
\end{proposition}
\begin{proof}
Let $k \in \N $.  We say that $ i \in [k,k^{2}+k] $ is a $k$-\emph{local regeneration}
point if the walk (started from 0) visits $i$ only once before getting to $k^{3}$. Let $X(k)$ be the collection of $k$-local regeneration points. Let $X'(k)$ be the collection of $k$-local regeneration points in $[k,k^2]$. We will show that 
\begin{equation}
\mathbb{E}[|X(k)|] \lesssim 1.\label{eq:simple}
\end{equation}
\begin{equation}
\mathbb{E}[|X(k)|\,|\,X'(k)\ne \eset]\gtrsim \log k.\label{eq:doob doob eyze khaboob}
\end{equation}

Comparing (\ref{eq:doob doob eyze khaboob}) to (\ref{eq:simple})
gives 
\begin{equation}
\label{e:1stmon}
\mathbb{P}[X'(k) \ne \eset] \le \frac{\mathbb{E}[|X(k)|]}{\mathbb{E}[|X(k)|\,|\,X'(k)\ne \eset]} \lesssim\frac{1}{\log k}.
\end{equation}
Summing  over all $k$ of the form $k=2^{2^{\ell}}$, the sum $\sum_{\ell=0}^{\infty}\mathbb{P}(X'(2^{2^{\ell}}) \ne \eset)  $ converges. Hence by the Borel-Cantelli Lemma there are a.s.\ only finitely
many local regeneration points  (where we have used the fact that $\cup_{\ell=0}^{\infty}[2^{2^{\ell}},(2^{2^{\ell}})^{2}] \cap \Z=\Z_{+} \setminus \{0,1 \}  $). But that means, of course, that there are only finitely many cut-times.

We first verify (\ref{eq:simple}) via a straightforward calculation. For every $\ell \in[k,k^{2}+k]$
the effective-resistance from $\ell$ to $k^{3}$, denoted by $\mathcal{R}_{\ell \leftrightarrow k^3} $, satisfies (uniformly in $k$ and $\ell \in [k,k^2+k]$) 
\[
\mathcal{R}_{\ell \leftrightarrow k^3}=\sum_{i= \ell}^{k^3}\frac{1}{\omega_{i} } \gtrsim  \frac{1}{\log \ell}-\frac{1}{\log(k^{3})}\gtrsim \frac{1}{\log k},
\]
while the resistance to $\ell+1$ is $\mathcal{R}_{\ell \leftrightarrow \ell+1}=\frac{1}{\omega_{\ell} } \lesssim \frac{1}{\ell\log^{2}k}$. Thus  \[\Pr_0[\ell\in X(k)]=P(\ell,\ell+1) \frac{\mathcal{R}_{\ell \leftrightarrow \ell+1}}{\mathcal{R}_{\ell \leftrightarrow k^3}} \lesssim \frac{1}{\ell\log k}.\]
Summing over $\ell$ from $k$ to $k^{2}+k$ gives (\ref{eq:simple}). We now prove (\ref{eq:doob doob eyze khaboob}).

Let $B_{\ell}$ be the event that $\ell $ is the minimal element in $X'(k)$ (assuming $X'(k)\ne \eset$). Conditioning on
this more restricted event makes the walk between time $T_{\ell}$ and $T_{k^{3}}$
a random walk (with the same edge weights as $R$) conditioned not to return to $\ell$ by time $T_{k^3}$.\footnote{In other words, in terms of $X'(k) \cap ( \ell ,k^2+k] $, conditioning on $\ell$ being the minimal element in $X'(k)$ is the same as conditioning on $\ell \in X'(k)$.} In particular, for $m \in (\ell,k^2+k]$ we have that \[ \Pr_0[m \in X(k) \mid B_{\ell}  ]=\Pr_0[m \in X(k) \mid \ell \in X'(k)  ]=P(m,m+1) \Pr_{m+1}[T_{m}>T_{k^{3}} \mid T_{\ell}>T_{k^3}   ].\] We now apply the  Doob's transform corresponding to $T_{\ell}>T_{k^3} $. In the one dimensional
case it takes a particularly simple form: it transforms the weight
of the edge between $m$ and $m+1$ from $\omega_{m}$ to $\hat \omega_m:= \omega_{m}p_{m}p_{m+1}$
where $p_{i}:=\Pr_{i}[T_{k^3}<T_\ell] $. In our case $\omega_{m} \asymp m\log^{2}m$
and 
\begin{equation}
\label{e:pm}
p_{m}= \frac{\mathcal{R}_{\ell \leftrightarrow m}}{\mathcal{R}_{\ell \leftrightarrow k^3}} \asymp \sum_{n=\ell}^{m}\frac{1}{n\log^{2}n\mathcal{R}_{\ell \leftrightarrow k^3}}\asymp \frac{1}{\mathcal{R}_{\ell \leftrightarrow k^3}\log^{2}k}\sum_{n=\ell}^{m}\frac{1}{n}\asymp \frac{\log (m/\ell)}{\mathcal{R}_{\ell \leftrightarrow k^3}\log^{2}k}.
\end{equation}
 We now wish to estimate the conditional probability that a given $m \in (\ell,k^2+k]$ is in
$X(k)$.  We restrict ourselves to $m\in( \ell,\ell+ \sqrt{ k}]$. For  $m\in( \ell,2\ell]$  we have $\log (m/\ell)  \asymp  (m-\ell)/\ell$.
Again we calculate resistances: the resistance $\widehat{\mathcal{R}}_{m \leftrightarrow m+1}$ from $m \in( \ell,2\ell]$ to $m+1$ (w.r.t.\ the Doob transformed chain) satisfies 
\[\widehat{\mathcal{R}}_{m \leftrightarrow m+1} =\frac{1}{\hat \omega_m}= \frac{1}{\omega_{m}p_{m}p_{m+1}}\asymp 
\frac{\mathcal{R}_{\ell \leftrightarrow k^3}^2\log^{4}k}{m(\log m)^{2}(\log(m/\ell))^{2}}\asymp \frac{\ell \mathcal{R}_{\ell \leftrightarrow k^3}^2\log^{4}k}{(\log \ell)^{2}(m-\ell)^{2}}.
\]
We estimate the effective-resistance in the transformed chain from $m \in (\ell ,  \ell + \sqrt{ k} ] $ to $k^{3}$, $ \widehat{\mathcal{R}}_{m \leftrightarrow k^3} $, by
\begin{align*}
\widehat{\mathcal{R}}_{m \leftrightarrow k^3} & = \sum_{n=m}^{k^3 -1}\frac{1}{\omega_{m}p_{m}p_{m+1}} \lesssim \sum_{n=m}^{k^{3}}\frac{\mathcal{R}_{\ell \leftrightarrow k^3}^2\log^{4}k}{n(\log n)^{2}(\log(n/\ell))^{2}}  \\
 & \lesssim \sum_{n=m}^{2\ell-1}\frac{ \ell^2 \mathcal{R}_{\ell \leftrightarrow k^3}^2\log^{4}k}{n(\log n)^{2}(n-\ell)^{2}}+\sum_{n=2\ell}^{k^{3}}\frac{\mathcal{R}_{\ell \leftrightarrow k^3}^2\log^{4}k}{n(\log n)^{2}} \lesssim  \frac{\ell \mathcal{R}_{\ell \leftrightarrow k^3}^2\log^{4}k}{(\log \ell)^{2}(m-\ell)}.
\end{align*}
Thus
\[\frac{\Pr_0[m\in X(k)\,|\,\ell \in X(k')]}{P(m,m+1)}=\Pr_{m+1}[T_{m}>T_{k^{3}} \mid T_{\ell}>T_{k^3}   ]=\frac{ \widehat{\mathcal{R}}_{m \leftrightarrow m+1}}{ \widehat{\mathcal{R}}_{m \leftrightarrow k^3}}
\gtrsim  \frac{1}{m-\ell}.
\]
Since $P(m,m+1) \gtrsim 1 $, summing over $m \in (\ell,\ell+\sqrt{k}] $ 
gives (\ref{eq:doob doob eyze khaboob}).

Fix some $M \in \N $. Denote the event that the edge $\{k,k+1 \}$ is crossed less than $2M+1$ times by $A_k$. Precisely the same calculation as \eqref{eq:simple}-\eqref{e:1stmon} shows that $ \lim_{N \to \infty} \Pr_{N}[\cup_{n=N}^{\infty}A_n]=\lim_{N \to \infty} \Pr_{0}[\cup_{n=N}^{\infty}A_n]=0  $. Now it is easy to see that by the Markov property,  for all $N>0$ we have that $\Pr_{0}[\cap_{n=0}^{\infty}A_n^{c}] \ge \Pr_{0}[\cap_{n=0}^{N-1}A_n^{c}]\Pr_{N}[\cap_{n=N}^{\infty}A_n^{c}]  $, where $A_n^{c} $ denotes the complement of the event $A_n$, and so   $\Pr_{0}[\cap_{n=0}^{\infty}A_n^{c}]>0$ as desired. To see this, let $\tau$ be the first time by which every edge $\{k,k+1\} $ for $k < N $ has been crossed at least $2M+1$ times. Conditioned on $\cap_{n=0}^{N-1}A_n^{c} $ we have that $\tau<\infty $ and that $R_{\tau} \le N $. Thus the walk will reach $N $ at some time $\hat \tau \ge \tau $. Ignoring possible earlier crossing of some of the edges  $\{k,k+1\} $ for $k \ge N $ (ones which occurred prior to $\hat \tau$) we see that $\Pr_{0}[\cap_{n=N}^{\infty}A_n^{c} \mid \cap_{n=0}^{N-1}A_n^{c}] \ge \Pr_{N}[\cap_{n=N}^{\infty}A_n^{c}]  $.      
\end{proof}

Recall the definitions of the functions $g_s(k):=k\prod_{i=1}^s \log_{*}^{(i)}(k) $ and $\log_{*}^{(i)} $ from \eqref{e:logloglog}. The proof of Theorem \ref{thm:4} shall be carried out by induction.  We now describe a mechanism for analyzing a birth and death chain with weights  $w^{(s+1)}(k-1,k) \asymp g_{s+1}(k)(\log_{*}^{(s+2)}k)^{2}$ using one with edge weights  $w^{(s)}(k-1,k) \asymp g_{s}(k)(\log_{*}^{(s+1)}k)^{2}$. This is the key behind the induction step.
\begin{definition}
\label{def:induced}
Let $R$ be a birth and death chain. Let $\mathbf{S}':=(S_n')_{n=0}^{\infty}$ be the \emph{non-lazy version of the chain induced on} $\mathcal{D}:=\{2^k:k \in \N \} \cup \{0\}$.  By this we
mean the following: let $\tau_{0} := \inf \{t \ge 0:R_t \in \mathcal{D} \} $ and inductively set $S_i':=R_{\tau_{i}} $ and  $\tau_{i+1}:=\inf \{t>\tau_i : R_{t} \in \mathcal{D} \setminus \{R_{\tau_i} \} \} $ for all $i \ge 0$. We transform $\mathbf{S}'$ into a birth and death  $\mathbf{S}=\mathbf{S}(R):=(S_n)_{n=0}^{\infty}$  by setting $S_n=0$ if $S_n'=0$ and $S_n=k$ if $S_n'=2^{k-1}$.
\end{definition}
\begin{lemma}
\label{lem:scales} Let $s \in \N$. Let $R$ be a birth and death chain with weights $w(k-1,k) \asymp g_{s}(k)(\log_{*}^{(s+1)}k)^{2}$.  Then the weights corresponding to $\mathbf{S}(R)$ satisfy  $w^{\mathbf{S}}(k-1,k) \asymp g_{s-1}(k)(\log_{*}^{(s)}k)^{2}$.
\end{lemma}
We omit the details of this easy calculation. As preparation for the induction basis we need the following corollary, whose proof uses a similar mechanism as the one behind the induction step (here with Proposition \ref{p::bd1} as the ``input").
\begin{corollary}
\label{cor:s=1}
Let $R$ be a birth and death chain  weights $w(k-1,k):=\omega_{k-1} \asymp g_{1}(k) (\log_*^{(2)}k)^2$.
Then for every fixed $\delta \in (0,1)$ there exists some $p=p(\delta)>0$ such that with probability at least $p$ the walk $R$ crosses each edge $\{k,k+1 \}$ at least $2 \lceil k^{\delta} \rceil +1$ times for all $k$. 
\end{corollary}
\begin{proof}
Fix some $\delta \in (0,1)$. Let $\mathbf{S}=\mathbf{S}(R)=(S_n)_{n=0}^{\infty}$ be as above. Let $M=M(\delta)$ be some sufficiently large constant, to be determined later. By Proposition \ref{p::bd1} and Lemma \ref{lem:scales} we have that with positive probability $\mathbf{S} $ crosses each edge at least $2M+1$ times. Let $G$ be this event. We condition on $G$.

 Observe that  each round-trip journey of $\mathbf{S}$ from $k+1$ to $k+2$ and then (eventually) back to $k+1$ (possibly visiting $k+2$ multiple times before returning to $k+1$) corresponds to a round-trip journey of $R$ from $2^{k}$ to $2^{k+1}$ and then (eventually) back to $2^{k}$ (possibly visiting $2^{k+1}$ multiple times before returning to $2^{k}$). Using the fact that all of the edge weights of the edges in $[2^k,2^{k+1}]$ are within constant factor from one another, it is easy to see that the probability that during such a round-trip journey  an edge $\{i,i+1\} $ in $[2^{k},2^{k+1}]$ is crossed less then $  2^{k\delta  +2}$ times is  $\asymp 2^{-k(1-\delta)} $ (uniformly in $i$ and $k$). We give the details of this easy calculation for the sake of completeness. 

We consider the cases $i \in [2^k,\sfrac{3}{2}2^k]$ and  $i \in (\sfrac{3}{2}2^k,2^{k-1})$. In the  former case, the   effective resistance from $i+1 $ to $2^{k+1}$ is $\sfrac{|2^{k+1}-(i+1)|}{\omega_i+\omega_{i-1}}\asymp  \sfrac{2^{k}}{\omega_i+\omega_{i-1}} $ and so  $|\{t<T_{2^{k+1}} :R_{t+1}=i,R_{t} =i+1\}|$ (for the chain $R$ started at $i+1$) has a geometric distribution with mean of order $2^k $. In the latter case, we consider the Doob's transform of the chain conditioned on $T_{2^k}<\infty$. As $p_{j}:=\Pr_j[T_{2^k}<\infty  ] \ge \Pr_j[T_{2^k}<T_{2^{k+2}}  ] $ is uniformly bounded from below for all $j \in [2^k,2^{k+1}] $, the edge weights  $\hat \omega_m:= \omega_{m}p_{m}p_{m+1}$
for $\{m,m+1\}$  in this  Doob's transform (see the paragraph preceding \eqref{e:pm}) are also within constant factor from one another for $m \in [2^k,2^{k+1}]$. It follows that conditioned on returning to $2^k$ after reaching $i$, the number of times the edge $\{i,i+1\}$ is crossed in the direction from $i$ to $i+1$, prior to the return time to $2^k$, also has a geometric distribution with mean of order $|i-2^{k}|\asymp2^k $.

 By the Markov property  this remain true even given $\mathbf{S}$. Hence, given $G$, the probability that  an edge $\{i,i+1\} $ in $[2^{k},2^{k+1}]$ is not crossed at least $  2^{k\delta  +2}$ times in any of the first $M$ round-trips of $\mathbf{S} $ from $k+1$ to $k+2$ and then (eventually) back  is  $\lesssim 2^{-kM(1-\delta)} $ (uniformly in $i$ and $k$). Taking $M \ge \frac{2}{1-\delta}$ and applying a union bound over all $ i \in [2^k,2^{k+1}-1]$ the proof is concluded in an analogous manner to the way the proof of Proposition \ref{p::bd1} was concluded. Namely, if $R$ and  $\mathbf{S} $ are coupled as above, $D$ is the event that every edge is crossed by   $\mathbf{S} $ at least $2M+1$ times, and $A_k $ is the event that the edge $\{k,k+1\} $ is crossed by $R$ less than  $2 \lceil k^{\delta} \rceil +1$ times,  then $\sum_{k \ge 0 } \Pr_0[A_k \mid D]< \infty $ and so $\lim_{N \to \infty } \Pr_0[\cap_{k \ge N } A_k^{c} \mid D]=1 $. Thus  $\lim_{N \to \infty } \Pr_0[\cap_{k \ge N } A_k^{c} ]>0 $.  Finally, as in the end of the proof of Proposition \ref{p::bd1}, we have that  $\Pr_{0}[ \cap_{n=0}^{\infty}A_n^{c} ] \ge \Pr_{0}[ \cap_{n=0}^{N-1}A_n^{c}]  \Pr_{N}[\cap_{n=N}^{\infty}A_n^{c}]$ for all $N$. As $\Pr_{0}[ \cap_{n=0}^{N-1}A_n^{c}]>0$ for all $N$, this concludes the proof.      
\end{proof}
For the induction step we need one additional lemma.
\begin{lemma}
\label{lem:KSSscales} Let $s \in \N$. Let $R$ be a birth and death chain with weights $w(k-1,k)\asymp g_{s}(k)(\log_{*}^{(s+1)}k)^{2}$. For every $L \in (0,\infty) $  there exists some $p=p(L,s)>0$ such that  for all $k$ the following hold 
\begin{itemize}
\item[(i)] Between any transition from $2^k$ to $2^{k+1}$, with probability at least $p$ every edge in $[2^{k},3 \cdot 2^{k-1}]$ is crossed at least $L \cdot 2^k $ times.
\item[(ii)] Between any transition from $2^{k+1}$ to $2^{k}$, with probability at least $p$ every edge in $[3 \cdot 2^{k-1},2^{k+1}]$ is crossed at least $L \cdot 2^k $ times.
\end{itemize}  
\end{lemma}
\begin{proof}
We only prove (i) as the proof of (ii) is analogous. Observe that the edge-weights in the interval $[2^k,2^{k+1}]$ are within constant factor of one another. The problem is thus reduced to the following setup. Consider a birth and death chain on $\{0,1,\ldots 2n\}$ started at $0$ whose edge weights satisfy (*) $1/C \le c(i-1,i) \le C$ for all $i \in [2n]$ (think of $n$ as $2^{k-1}$). Let 
\begin{equation}
\label{e:k65}
K_i=K_i^{(n)}:=|\{t < T_{2n}:(X_{t},X_{t+1})=(i+1,i) \}|
\end{equation}
 be the number of transitions from $i+1$ to $i$ prior to $T_{2n}$. We need to show that for every $C \ge 1$ (as in (*)) and $L>0 $ there exists some $p=p(C,L)>0$ such that the probability that $\min \{K_i^{(n)}:0 \le i<n \} \ge Ln$ is at least $p$ (for all $n \in \N $).   

We now argue that it suffices to consider the case that $c(i-1,i)=1$  for all $i<2n$ while $1 \le c(2n-1,2n) \le 3$. Indeed, by subdividing some edges if necessary (cf.\ \cite{traces}) we can find a sequence $x_0=0,x_1,\ldots x_{2m}=2n$, where $m \ge  n/(2C^2) - 2$ such that (1) the distance of $x_i$ from $0$ is increasing in $i$ and (2) the effective-resistance between $x_{i-1}$ and $x_i$ is exactly $C$ for all $i <2m$ while for $i=2m $ it is between $C$ and $3C$.   Instead of studying  $\mathbf{K}:=(K_j)_{j=0}^{n}$  we may consider the number of transitions between these points. To be precise, for every edge $\{i,i+1\}$ which was not subdivided and is in the interval between $x_j$ and $x_{j+1}$, we have that every journey of the new chain from $x_j$ to $x_{j+1}$, or vice-versa, corresponds to a crossing of $\{i,i+1\}$ in the original chain. This is not the case for edges that were subdivided. However if $\{i,i+1\}$ was subdivided, with $x_j$ added between $i$ and $i+1$, every journey of the new chain from $x_{j+1}$ to $x_{j-1}$ (or vice-versa) does correspond to a crossing of  $\{i,i+1\}$ in the original chain. Here, by a ``journey from  $x_{j+1}$ to $x_{j-1}$" we mean that the walk was at $x_{j+1}$, reached $x_j$ and then reached $x_{j-1}$ before returning to $x_{j+1}$. The notion of a ``journey from  $x_{j-1}$ to $x_{j+1}$" is analogously defined. We note that this notion is not the same as our usage of the term ``a journey from $0$ to $2n$" in which we do not require the walk to not return to $0$ before reaching $2n$.   

By  performing a network reduction, collapsing the interval between $x_{j}$ and $x_{j+1}$ to a single edge, we obtain an auxiliary birth and death chain with edge weights as described above. Let us label $x_j$ by $j$, and let us also denote the length of the corresponding interval by $2n$, rather than $2m$ (as opposed to the initial setup, now the edge weights are all equal to  1, apart perhaps from the one between $2n-1$ and $2n$ which is between $1$ and $3$). By the discussion above, it suffices to argue that for every $L>0$ there exists some $p'(L)>0$ (independent of $n$) such that with probability at least $p'(L)$, during a single journey from $0 $ to $2n$, for all $i  \in [0, n] $  the number of journeys between $i$ and $i+2$ (i.e.\ an $i$ to $i+1$ transition followed by an $i+1$ to $i+2$ transition, or vice-versa) is at least $L n$. 

We now argue that in  fact it suffices to show  that
for every $L>0$ there exists some $p''(L)>0$ (independent of $n$) such that with probability at least $p''(L)$, during a single journey from $0 $ to $2n$, for all $i  \in [0, n] $  the number of crossings of  $\{i,i+1\}$ is at least $L n$. Indeed, there exist absolute constants $c,c',\hat C>0$ such that for all $n$ and all $i \in [0,n]$ conditioned on the number of crossing of  $\{i,i+1\}$ during a single journey from $0 $ to $2n$ being at least $L n$, the conditional probability that the number of journeys between $i$ and $i+2$ (in the above sense) during a single journey from $0 $ to $2n$ is at most $c'L n $ is at most $\hat C e^{-cL n}$ (in fact, we can take $c'$ to be any number in $(0,1/4)$). We omit the details of this calculation. This concludes the justification of the reduction to the case where the edge weights are as described above. 

We now turn to the analysis of the probability that $\min \{K_i^{(n)}:0 \le i<n \} \ge Ln$, where $K_i^{(n)}$ is as in \eqref{e:k65}.  
Fix some $L \ge 1$. Observe that $K_0^{(n)}$ follows a Geometric distribution with mean $\asymp n $ and hence we may condition on the event that $K_0^{(n)} \ge 3Ln$ (as the corresponding probability is bounded from below, uniformly in $n$ for each fixed $L$).          

To analyze $\mathbf{K} $ when the edge weights are all equal to 1 (apart perhaps from the last one, which is between 1 and 3)  we use a Kesten-Kozlov-Spitzer type argument \cite{KKS}. In other words, we study the process $\mathbf{K}$ as a time-inhomogeneous Markov process.\footnote{In fact, it is a time-inhomogeneous branching process with time-inhomogeneous immigration.} Observe that $K_{j+1}$ can be decomposed into a sum of $K_j+1$ terms as follows:
\begin{equation}
\label{e:decomposition0}
K_{j+1}=J_{j+1}+\sum_{\ell=1}^{K_j}\xi_\ell^{(j+1)},
\end{equation}
where $\xi_\ell^{(j+1)}$ is the number of transitions from $j+2$ to $j+1$ between the $(2\ell-1)$th and $(2\ell)$th crossings of $\{j,j+1\}$  which occurred prior to $T_{2n}$  (i.e.\ during the $\ell$th excursion in $[j,2n)$ from $j$ to $j$ prior to $T_{2n}$) and $J_{j+1}$ is the number of transitions from $j+2$ to $j+1$ after the last visit to $j$ prior to $T_{2n}$. Given $K_j=m$ we have that $\xi_1^{(j+1)}+1,\xi_2^{(j+1)}+1,\ldots, \xi_m^{(j+1)}+1 $ are i.i.d.\ Geometric random variables whose mean $\alpha_{j+1}=\alpha_{j+1}(n) $ depends on $j$ but not on $m$, while the law of $J_{j+1}+1$, which is also Geometric of some mean  $\beta_{j+1}=\beta_{j+1}(n)$, is independent of $K_j$.  Using effective-resistance calculations in the corresponding Doob transformed chains, it is easy to verify that $2-\alpha_j \asymp \frac{1}{n} $ and that $\beta_j-1 \gtrsim 1  $ (for concreteness, $\beta_j > 9/8$) uniformly in $n$ and $j \in [n]$. 

For our purposes we may ignore the terms $J_j$. That is, we may consider  a subcritical branching process $B_0,B_1,\ldots$ which has initial population size $B_0= \lceil 3Ln \rceil $ (recall that we may assume that $K_0^{(n)} \ge 3Ln$) and offspring distribution $\xi$, where $\xi +1 \sim  $ Geometric with mean $2-C_0/n$, where $C_0$ is some constant chosen so that $ 2-C_0/n\le \min_{j \in [n] } \alpha_j(n) $.  It suffices to show that with probability bounded from below (uniformly in $n$) we have that up to the $n$-th generation all generations are of size at least $Ln$.

Instead of studying the above subcritical branching process, we study a related random walk. One may label all of the individuals in all generations of the branching process by the set $\N$ in a manner that the generation number to which the individual labeled $j$ belongs to is non-decreasing in $j$ (i.e., the individuals from the $i$-th generation are labeled by $B_0+B_1+\cdots +B_{i-1}+1,\ldots,B_{0}+B_1+\cdots +B_{i}  $). We couple the branching process $B_0,B_1,\ldots,$ with a random walk $Z_j(n):= \sum_{i=1}^{j} Y_{i}(n) $, where $Y_{1}(n)+2,Y_2(n)+2,\ldots  $ are i.i.d.\ Geometric random variables with mean  $2- C_0/n$ as follows. We take the number of offspring of individual $j$ to be $Y_j(n)+1 $. Then $Z_{B_0+B_1+\cdots +B_j}(n)=B_{j+1}-B_{0} $ for all $j \ge 0$. Thus it suffices to argue that for all fixed $C_1,L \ge 1 $ the walk $(Z_j(n))_{j=0}^{\lceil C_1 (Ln)^2 \rceil}$ remains in $[-2Ln,2Ln]$ with probability bounded from below uniformly in $n$. Indeed, if this occurs for $C_1=3 $ and $L \ge 1$ then necessarily $B_0+B_1+\cdots +B_n \le  C_1 (Ln)^2   $.

  Observe that $S_n(t):=\frac{1}{n} [(n^{2}t-\lfloor n^{2}t \rfloor)Z_j(\lceil n^{2}t \rceil )+(\lceil n^{2}t \rceil-n^{2}t)Z_j(\lfloor n^{2}t \rfloor )  ] $ converges in distribution to a Brownian motion with a constant drift and some constant variance, where $t \in [0,C_1L^{2}]$. Let $W(t)$ be a Brownian motion with such a drift. By the Cameron-Martin Theorem the laws of a Brownian motion and of a Brownian motion with the same variance and with some fixed drift (both taken in some finite time interval) are mutually absolutely continuous.   It follows that for every fixed $C_1$, the probability that $W(t) \in [-L,L]$ for all $t \in [0,C_1L^{2}] $ is positive.   Thus the probability that $S_n(t) \in [-2L,2L]$ for all $t \in [0,C_1L^{2}] $ is bounded from below, uniformly in $n$. 
\end{proof}

\subsection{Proof of Theorem \ref{thm:ata melagleg alay}}

\emph{Proof:}
Let $s \in \Z_+$. Let $\mathbf{S}=\mathbf{S}(R)$ be as in Definition \ref{def:induced}. By Lemma \ref{lem:scales}  the edge weight between $k$ and $k+1$ w.r.t.\ $\mathbf{S}$ is $\asymp g_{s+1}(k)(\log_{*}^{(s+2)}k)^{2} $. We argue by induction on $s$. We start with the case $s=0$. In this case,  by Corollary \ref{cor:s=1} with positive probability $\mathbf{S} $ crosses every edge $\{k,k+1\}$ at least $2 \sqrt{k}+1 $ times.  Let $G$ be
this event. We condition on $G$.  We call a round-trip journey of $\mathbf{S}$ from $k+1$ to $k+2$ and then (eventually) back to $k+1$ (possibly returning to $k+2$ multiple times prior to that) \emph{successful} if for all $i \in [2^k,2^{k+1}-1]$ the edge $\{i,i+1\}$ is crossed at least $2^k$ times by $R$ (during its corresponding round-trip journey from $2^k$ to $2^{k+1}$ and then back).

 By Lemma \ref{lem:KSSscales} there exists some $p$  (the same $p$ for all $k$) such that each journey is successful with probability at least $p$.   Moreover, this holds even if we condition
on $\mathbf{S}$. In particular, conditioned on $G$ the probability that there are less than $\half p \sqrt{k} $ successful journeys from $k+1$ to $k+2$ and then back, decays exponentially in $\sqrt{ k}$. Similarly to the way the argument in the proof of Proposition \ref{p::bd1} is concluded, this implies that with positive probability, for all $k$ there are at least  $\half p \sqrt{k} $ such successful journeys. On this event, with positive probability $R$ crosses every edge $\{i,i+1\}$ at least $c i \sqrt{\log (i+3)} > c i  $ times. This concludes the induction basis.

The induction step from $s$ to $s+1$ looks exactly the same, apart from the fact that instead of the term $\sqrt{k} $, when we use the induction hypothesis for $s$, the corresponding term we get is $c(s)g_s(k) $. The same reasoning yields that with positive probability, for all $k$ and all $i \in [2^k,2^{k+1}-1]$ we have that the edge $\{i,i+1\}$ is crossed at least  $2^k \times c(s+1)g_s(k) \asymp g_{s+1}(i) $ times. We omit the details.   
\qed
\section{Proof of Theorem \ref{thm:3}}
\label{s:proofof3}
Let $Q$ be the transition kernel of the birth and death chain with edge weights $w(k,k+1)=g_4(k)(\log_{*}^{(5)}k)^{2} $, where $g_s(k)$ is as in \eqref{e:logloglog}. Let $\|(x,y) \|_{\infty} :=\max \{|x|,|y| \} $.
Let \[S(k):=\{z \in \Z^2:\|z \|_{\infty} =k \} \quad \text{and} \quad C(k):=\{(\pm k,k),(\pm k,-k) \}\] be the $k$-sphere (w.r.t.\ $\| \bullet\|_{\infty}$) and its four corners, respectively.
Consider a Markov chain $\X=(X_n)_{n=0}^{\infty}$ on $\Z^2$ which moves according to the following rule:
\begin{itemize}
\item
From $(0,0)$ it moves to each $x \in S(1) \setminus C(1) $ with probability $1/4$.
\item When the chain is at $C(k)$ for $k>0$, it moves counter-clockwise inside $S(k)$ w.p.\ 1.
\item Whenever the chain is at $S(k) \setminus C(k)$ for $k>0$  it moves counter-clockwise inside $S(k)$ with probability $1-\frac{1}{k^2}$ and  with probability $\frac{Q(k,k+1)}{k^2}$ (respectively, $\frac{Q(k,k-1)}{k^2}$) it moves to the (unique) adjacent (w.r.t.\ $\Z^2$) vertex in $S(k+1)$ (respectively, $S(k-1)$).
\end{itemize}
By construction, if we view $\|X_{k}\|_{\infty}$ only at times at which it changes its value we get a copy of the birth and death chain corresponding to $Q$. In particular, $\X$ is transient. By Theorem \ref{thm:4}, with positive probability the chain $\X$ makes at least $cg_{2}(k) $ transitions from $S(k+1)$ to $S(k)$ for all $k$. Denote this event by $G$. By construction, at each transition from  $S(k+1)$ to $S(k)$ the probability that a certain edge connecting  $S(k+1)$ to $S(k)$ is crossed is at least $\gtrsim \frac{1}{k} $. Thus conditioned on $G$ the chance that a certain edge connecting  $S(k+1)$ to $S(k)$ is not crossed is at most $(1-\frac{c_0}{k})^{cg_{2}(k)} \le \exp [- \Omega ((\log k )  \log \log k) ]$. Since there are only $O(k)$ such edges, the conditional probability (given $G$) that one of them is not crossed is at most $\exp [- \Omega ((\log k) \log \log k) ]$, which is summable in $k$. From this it is easy to see that the probability that each edge is crossed is positive. Indeed the probability that every edge which is not incident to some vertex in $\cup_{i=0}^{N-1}S(i) $ is crossed tends to 1 as $N \to \infty$. For each fixed $N$, with positive probability the chain crosses all edges which are  incident to some vertex in $\cup_{i=0}^{N-1}S(i) $ and then returns to the origin. Applying the Markov property concludes the proof. \qed
\subsection{Comments concerning the construction:}
\label{s:conc}
\begin{itemize}
\item[(1)] The conclusion that with positive probability every edge is crossed can be strengthen to the claim that with positive probability for all $k$ every edge which is adjacent to some vertex in $S(k)$ is crossed at least $\log k \log \log k $ times.  
\item[(2)] The reason the corners $C(k)$ are treated separately is that it is impossible to move from $C(k)$ to $S(k-1)$, and so in order to avoid a bias for $(\|X_n\|_{\infty})_{n=0}^{\infty}$ we do not allow transitions from $C(k)$ to $S(k+1)$ (transitions in the opposite direction are allowed).
\item[(3)] We could have defined the transition probabilities inside each $S(k)$ (conditioned that the chain moves inside $S(k)$) to be such that the chain moves to each of the two neighbors in $S(k)$ of the current position with equal probability (note that this uses the fact that the probability of staying within a $k$-sphere is sufficiently close to 1). However, this example would  still not be reversible due to the corners $C(k)$ (see the previous comment).

\item[(4)]
The following modification to the above example should still exhibit the desired behavior, while having transition probabilities uniformly bounded from below: Whenever the chain is at $S(k) \setminus C(k)$ for $k>0$  let it move counter-clockwise inside $S(k)$  with probability $1/3$, clockwise inside $S(k)$ with probability $1/6$ and  with probability $\frac{Q(k,k+1)}{2}$ (respectively, $\frac{Q(k,k-1)}{2}$) it moves to the (unique) adjacent vertex (w.r.t.\ $\Z^2$) in $S(k+1)$ (respectively,  $S(k-1)$).  Whenever the chain is at $C(k)$ for $k>0$  let it move counter-clockwise inside $S(k)$ with probability $2/3$ and clockwise inside $S(k)$ with probability $1/3$.   In fact, it should be the case that in this modification, with positive probability, for all $k$ every edge which is incident to some vertex in $S(k)$ is crossed in both directions at least $\gtrsim \log k \log \log k$ times, apart from the edges between $C(k)$ and $S(k+1)$, which are only crossed in the direction from $S(k+1)$ to $C(k)$.

We now provide a very rough proof sketch. This sketch should be considered more as a heuristic than a real proof. First argue that the stationary measure of this walk is pointwise within a constant factor of $\pi(x):=w(\|x \|_{\infty},\|x \|_{\infty}-1)+w(\|x \|_{\infty},\|x \|_{\infty}+1)$, where $w$ are the edge-weights corresponding to the birth and death chain $Q$.\footnote{Observe that if $x $ is of distance at least 2 from a corner then $\pi P (x)=\pi(x) $.} In particular, for each $k$ its value on $S(k)$ varies by at most some constant factor. 

For each $k$, consider the Doob transform of the chain corresponding to conditioning on returning to $S(k)$. We now consider as an auxiliary chain (on $S(k)$) the aforementioned Doob transformed chain at times it returns to $S(k)$.\footnote{In practice, one has to justify the fact that conditioned on $G$ (where $G$ is as in the proof of Theorem \ref{thm:3}), when considering the first  $\half cg_{2}(k)$ returns to $S(k)$, we can couple them with those of the Doob transformed chain with probability sufficiently close to 1.}
 It is not hard to show that its mixing-time is $O( k)$ (in fact it is $\Theta (\sqrt{k})$, where here time is measured by the number of returns to $S(k)$) and that the expected number of visits to each point in $S(k)$ by the $k$-th return to $S(k)$ is $ O( 1)$. By the aforementioned mixing time estimate (and the fact that the stationary distribution of the auxiliary chain is  point-wise within a constant factor from the uniform distribution on $S(k)$), the expected number of visits to each point in $S(k)$ by the $\lceil Lk \rceil $-th return to $S(k)$ is actually $ \Theta( 1)$ (for some constant $L \ge 1 $).  From this one can further deduce that the probability of hitting any fixed point in $S(k)$ by the $\lceil Lk \rceil $-th return is $ \gtrsim 1$.\footnote{The probability a point is visited equals the ratio of the expected number of visits and the expected number of visits, conditioned that this number is positive. Both quantities are $\Theta(1)$.} Thus the probability a certain point in $S(k)$ is not hit at least $c_0 \log k \log \log k $ (where $k \ge e^e$) times by the $\half cg_{2}(k) $-th return to $S(k)$ is exponentially small in $\log k \log \log k$. From here it is not hard to conclude the proof. \item[(5)]
Despite the fact that in the example described at (4) with positive probability for all $k$ each vertex in $S(k)$ is visited $\gtrsim \log k \log \log k$ times, we believe that the expectation of the size of the range of this Markov chain grows linearly in time. By the Paley-Zygmund inequality this implies that the probability that the size of the range by time $t$ is at least half its mean (which is $\gtrsim t $) is at least some constant $c>0$ (indeed, the second moment of the range is at most $t^2$, while the square of its first moment  is $\gtrsim t^2$).  

Indeed by time $t$ the chain  typically spends a fraction of the time at distance $\gtrsim \sqrt{ t} $ from the origin. We strongly believe that  if $\mathrm{dist}(x,\mathbf{0}) \gtrsim \sqrt{t} $, then for the Markov chain from (4) we have that $\sum_{i=0}^{t}P^i(x,x) \lesssim 1$, from which one can deduce that the expectation of the size of the range indeed grows linearly in time.

 To see this, fix some small $\eps \in (0,1/10)$. For all $i \in [\half t,t]$ let $D_i$ be the event that $\mathrm{dist}(X_{i},\mathbf{0}) \ge \eps  \sqrt{t} $ and that  the number of visits to $X_i$ between time $i$ and $t$ is at least $\eps^{-2}\sup_{x:\mathrm{dist}(x,\mathbf{0}) \ge \eps \sqrt{ t}  }\sum_{j=0}^{t}P^j(X_{i},X_{i}) \le C(\eps)$. Then the probability of $D_i$ is at most $\eps^2$ and by Markov's inequality the probability that   $ \sum_{i= \half t }^t 1_{D_{i}} \ge \half \epsilon t $  is at most $\eps$. Since the probability that during the time period $[\half t,t]$ the walk spends at least half its time at distance at least  $\eps \sqrt{ t}$ from the origin is close to 1  when $\eps$ is small (denote this probability by $p$), we get that with probability bounded from below (namely, at least $p-\eps$) there are at least $(\frac{1}{4} -\eps) t $ times $i$ at which the number of returns to $X_i$ between time $i$ and $t$ is at most some constant $C(\eps)$. On this event, clearly the size of the range is of size at least $\frac{t}{8C(\eps)} $.
\item[(6)]
We note that if from every vertex in $S(k) \setminus C(k) $ the probabilities of moving clockwise and counter-clockwise are both taken to be a $1/4$ (and otherwise with probability $1/2$ the chain moves between spheres according to $Q$)  it would have required $\Omega( k \log^2 k )$ transitions from either $S(k+1)$ or $S(k-1)$ to $S(k)$ in order for every state in $S(k)$ to be visited.\footnote{This can be derived from the fact that every time some $v \in S(k)$ is visited, the expected number of returns to $v$ during the following $k$ steps is $\Omega (\log k)$.} However, by Theorem \ref{thm:4} there is no transient birth and death chain that crosses for each $k$ either $\{ k,k+1\}$ or  $\{ k-1,k\}$  at least $\Omega( k \log^2 k )  $  times with positive probability. Thus even if one overcomes the difficulty which arises from the existence of the corners, the example from the proof of Theorem \ref{thm:3} still cannot be transformed into a reversible chain with transition probabilities bounded from below.
\end{itemize}
\section{Relaxing the assumption that $\X$ admits a stationary measure}
\label{s:relax}
\begin{claim}
\label{reduction}
Let $\X$ be an irreducible transient Markov chain on a countable state space $V$ with transition kernel $P$. Let $\overrightarrow{G}(\X):=(V,\{(x,y) \in V^2:P(x,y)>0 \})$. Assume that one can delete from $\overrightarrow{G}(\X) $ some directed edges so that the in-degree of each site in the obtained graph is finite, and so that the obtained graph contains a directed path from $x$ to $y$ for every $x,y \in V$.   Then a.s.\ the trace of $\X$ is recurrent for SRW.
\end{claim}
\begin{proof}
Label $V$ by $v_1,v_2,\ldots$. Let  $\eps \in (0,1)$. Denote $E_{i}^{\mathrm{in}}:=\{(y,v_i):P(y,v_i)>0 \} $.  By transience, the expected number of visits to every vertex is finite. This means that for every $v_i$ there are collections $F_i^{\mathrm{in}}   \subset E_{i}^{\mathrm{in}} $ such that (1) $  J_{i}^{\mathrm{in}}:=E_{i}^{\mathrm{in}} \setminus F_{i}^{\mathrm{in}} $ is finite for all $i$ and  (2) the expected number of times (including multiplicities) that all of the directed edges in $ F_{i}^{\mathrm{in}}  $ are crossed is at most $\epsilon 2^{-(i+1)}$ (i.e., $\mathbb{E}_o[| \{t \in \Z_+ :(X_t,X_{t+1}) \in F_{i}^{\mathrm{in}} \} |] \le \eps 2^{-(i+1)} $). 

Denote $\tilde J:=\bigcup_{i \in \N} J_{i}^{\mathrm{in}}  $.
Now let $\widetilde \X $ be the Markov chain with $G(\widetilde \X)=(V,\tilde J) $ whose transition probabilities are given by $\widetilde P(v_i,y):=\frac{\Ind{(v_i,y) \in \tilde J  }P(v_{i},y)}{\sum_{u :\, (v_i,u) \in \tilde J  }P(v_i,u)} $. By the assumption on  $\overrightarrow{G}(\X) $, by increasing $J_{i}^{\mathrm{in}}  $ for some $i$ (so that it is still finite), if necessary, we may ensure that  $\widetilde \X $ is irreducible. By Harris \cite{Harris} and the finiteness of the $ J_{i}^{\mathrm{in}}  $-s,   $\widetilde \X $ admits a stationary measure.    

 By construction, it is clear that one can couple $\X$ and $\tilde \X$ (both starting from $o$) so that with probability at least $1-\eps$ they follow the same trajectory (i.e., $X_t=\widetilde X_t $ for all $t$). It follows that $\widetilde \X $ returns to $o$ finitely many times with positive probability, and hence is transient. Hence, by Theorem \ref{thm:1} its trace is  a.s.\  recurrent for the SRW. It follows that the same holds for $\X$ with probability at least $1-\eps$. The proof is concluded by sending $\eps$ to $0$. 
\end{proof}

\section*{Acknowledgements}
The examples from the proofs of Theorems \ref{thm:3} and \ref{thm:4} are due to Gady Kozma. We are tremendously grateful to him for allowing us to present them. We would also like to thank  Tom Hutchcroft and Russell Lyons for  helpful suggestions.

\bibliographystyle{plain}
\bibliography{traces}

\begin{thebibliography}{10}

\bibitem{aldous}
David Aldous and Jim Fill.
\newblock Reversible {M}arkov chains and random walks on graphs, 2002.

\bibitem{ALP}
Alano Ancona, Russell Lyons, and Yuval Peres.
\newblock Crossing estimates and convergence of {D}irichlet functions along
  random walk and diffusion paths.
\newblock {\em Ann. Probab.}, 27(2):970--989, 1999.

\bibitem{net}
M{\'a}rton Bal{\'a}zs and Aron Folly.
\newblock An electric network for nonreversible {M}arkov chains.
\newblock {\em American Mathematical Monthly}, 123(7):657--682, 2016.

\bibitem{BC}
Itai Benjamini and Nicolas Curien.
\newblock Recurrence of the {$\Bbb Z^d$}-valued infinite snake via
  unimodularity.
\newblock {\em Electron. Commun. Probab.}, 17:no. 1, 10, 2012.

\bibitem{BGG}
Itai Benjamini and Ori Gurel-Gurevich.
\newblock Almost sure recurrence of the simple random walk path.
\newblock {\em Unpublished manuscript. arXiv preprint math/0508270}, 2005.

\bibitem{traces}
Itai Benjamini, Ori Gurel-Gurevich, and Russell Lyons.
\newblock Recurrence of random walk traces.
\newblock {\em Ann. Probab.}, 35(2):732--738, 2007.

\bibitem{cut3}
Itai Benjamini, Ori Gurel-Gurevich, and Oded Schramm.
\newblock Cutpoints and resistance of random walk paths.
\newblock {\em Ann. Probab.}, 39(3):1122--1136, 2011.

\bibitem{BRW}
Itai Benjamini and Sebastian M\"uller.
\newblock On the trace of branching random walks.
\newblock {\em Groups Geom. Dyn.}, 6(2):231--247, 2012.

\bibitem{BS}
Itai Benjamini and Oded Schramm.
\newblock Recurrence of distributional limits of finite planar graphs.
\newblock In {\em Selected works of {O}ded {S}chramm. {V}olume 1, 2}, Sel.
  Works Probab. Stat., pages 533--545. Springer, New York, 2011.

\bibitem{doyle}
Peter~G Doyle and J~Laurie Snell.
\newblock {\em Random walks and electric networks}.
\newblock Mathematical Association of America,, 1984.

\bibitem{doyle2}
Peter~G Doyle and Jean Steiner.
\newblock Commuting time geometry of ergodic {M}arkov chains.
\newblock {\em arXiv preprint arXiv:1107.2612}, 2011.

\bibitem{nonrev}
Alexandre Gaudilli{\`e}re and Claudio Landim.
\newblock A {D}irichlet principle for non reversible {M}arkov chains and some
  recurrence theorems.
\newblock {\em Probability Theory and Related Fields}, 158(1-2):55--89, 2014.

\bibitem{Harris}
T.~E. Harris.
\newblock Transient {M}arkov chains with stationary measures.
\newblock {\em Proc. Amer. Math. Soc.}, 8:937--942, 1957.

\bibitem{NP}
N.~James and Y.~Peres.
\newblock Cutpoints and exchangeable events for random walks.
\newblock {\em Teor. Veroyatnost. i Primenen.}, 41(4):854--868, 1996.

\bibitem{JLP}
Nicholas James, Russell Lyons, and Yuval Peres.
\newblock A transient {M}arkov chain with finitely many cutpoints.
\newblock In {\em Probability and Statistics: Essays in Honor of David A.
  Freedman}, pages 24--29. Institute of Mathematical Statistics, 2008.

\bibitem{KKS}
H.~Kesten, M.~V. Kozlov, and F.~Spitzer.
\newblock A limit law for random walk in a random environment.
\newblock {\em Compositio Math.}, 30:145--168, 1975.

\bibitem{cut}
Gregory~F. Lawler.
\newblock Cut times for simple random walk.
\newblock {\em Electron. J. Probab.}, 1:no.\ 13, approx.\ 24 pp.\, 1996.

\bibitem{levin}
David~Asher Levin, Yuval Peres, and Elizabeth~Lee Wilmer.
\newblock {\em Markov chains and mixing times}.
\newblock American Mathematical Soc., 2009.

\bibitem{lyons}
Russell Lyons and Yuval Peres.
\newblock {\em Probability on trees and networks}, volume~42.
\newblock Cambridge University Press, 2016.

\bibitem{morris}
Ben Morris.
\newblock The components of the wired spanning forest are recurrent.
\newblock {\em Probab. Theory Related Fields}, 125(2):259--265, 2003.

\bibitem{NW}
C.~St. J.~A. Nash-Williams.
\newblock Random walk and electric currents in networks.
\newblock {\em Proc. Cambridge Philos. Soc.}, 55:181--194, 1959.

\bibitem{pitman}
Jim Pitman.
\newblock {\em Probability}.
\newblock Springer, {N}ew-{Y}ork, 1993.

\end{thebibliography}

\end{document}